\documentclass{amsart}
\usepackage{hyperref}
\usepackage{amssymb}
\usepackage{amsrefs} 
\usepackage[all]{xy}

\newtheorem{theorem}{Theorem}[section]
\newtheorem{definition}[theorem]{Definition}
\newtheorem{proposition}[theorem]{Proposition}
\newtheorem{lemma}[theorem]{Lemma}

\theoremstyle{remark}
\newtheorem*{remark}{Remark}

\DeclareMathOperator{\Aut}{Aut}
\DeclareMathOperator{\disc}{disc}
\DeclareMathOperator{\End}{End}
\DeclareMathOperator{\Gal}{Gal}
\DeclareMathOperator{\Hom}{Hom}
\DeclareMathOperator{\Jac}{Jac}
\DeclareMathOperator{\Norm}{Norm}
\DeclareMathOperator{\Pic}{Pic}
\DeclareMathOperator{\rank}{rank}
\DeclareMathOperator{\Tr}{Tr}
\DeclareMathOperator{\tworank}{2-rank}

\newcommand{\CC}{{\mathbf{C}}}
\newcommand{\EE}{{\mathbf{E}}}
\newcommand{\FF}{{\mathbf{F}}}
\newcommand{\QQ}{{\mathbf{Q}}}
\newcommand{\ZZ}{{\mathbf{Z}}}

\newcommand\gothA{{\mathfrak{A}}}
\newcommand\gothp{{\mathfrak{p}}}
\newcommand\gothq{{\mathfrak{q}}}

\newcommand{\calB}{{\mathcal{B}}}
\newcommand{\calC}{{\mathcal{C}}}
\newcommand{\calO}{{\mathcal{O}}}
\newcommand{\calS}{{\mathcal{S}}}
\newcommand{\calU}{{\mathcal{U}}}

\newcommand{\col}{\,{:}\,}
\newcommand{\eps}{\varepsilon}
\newcommand{\Fq}{\FF_q}
\newcommand{\Sha}{\textup{III}\llap{\rule[-0.02ex]{1.1em}{0.1ex}\hspace{0.02em}}}

\newcommand{\Drinfeld}{Drinfel{\cprime}d}
\newcommand{\Vladut}{Vl\u adu\c t}

\renewcommand{\hat}{\widehat}

\newcommand\lowtilde{\lower0.7ex\hbox{\textasciitilde}}

\newcommand\us{\textunderscore}

\newcommand{\mybar}[1]{
  \mathchoice
  {#1\llap{$\overline{\phantom{\displaystyle\rm#1}}$}}
  {#1\llap{$\overline{\phantom{\textstyle\rm#1}}$}}
  {#1\llap{$\overline{\phantom{\scriptstyle\rm#1}}$}}
  {#1\llap{$\overline{\phantom{\scriptscriptstyle\rm#1}}$}}
}  
\renewcommand{\bar}{\mybar}

\newcommand{\bbar}{\bar{b}}
\newcommand{\rbar}{\bar{r}}
\newcommand{\xb}{\bar{x}}
\newcommand{\yb}{\bar{y}}
\newcommand{\pibar}{\bar{\pi}}
\newcommand{\zetabar}{\bar{\zeta}}
\newcommand{\Fqbar}{\bar{\FF}_q}
\newcommand{\gothpbar}{\bar{\gothp}}
\newcommand{\gothqbar}{\bar{\gothq}}

\begin{document}

\title[Bounding the number of points on curves]
{New methods for bounding the number of points on curves over finite fields}

\author[Howe]{Everett W. Howe}
\address{Center for Communications Research,
         4320 Westerra Court,
         San Diego, CA 92121-1967, USA.}
\email{however@alumni.caltech.edu}
\urladdr{http://www.alumni.caltech.edu/\lowtilde{}however/}

\author[Lauter]{Kristin E. Lauter}
\address{Microsoft Research,
         One Microsoft Way,
         Redmond, WA 98052, USA.}
\email{klauter@microsoft.com}

\date{28 February 2012}



\begin{abstract}
We provide new upper bounds on $N_q(g)$, the maximum number of rational
points on a smooth absolutely irreducible genus-$g$ curve over $\Fq$, 
for many values of $q$ and $g$.  Among other results, we find that 
$N_4(7) = 21$ and $N_8(5) = 29$, and we show that a genus-$12$ curve 
over $\FF_2$ having $15$ rational points must have characteristic 
polynomial of Frobenius equal to one of three explicitly given 
possibilities.

We also provide sharp upper bounds for the lengths of the shortest
vectors in Hermitian lattices of small rank and determinant over the
maximal orders of small imaginary quadratic fields of class number~$1$.

Some of our intermediate results can be interpreted in terms of 
Mordell--Weil lattices of constant elliptic curves over one-dimensional
function fields over finite fields.  Using the Birch and Swinnerton-Dyer
conjecture for such elliptic curves, we deduce lower bounds on the
orders of certain Shafarevich--Tate groups.
\end{abstract}

\maketitle

\section{Introduction}
\label{S:intro}

The last three decades have seen increasing interest in the calculation
of the value of $N_q(g)$, the maximum number of rational points on a
smooth, absolutely irreducible curve $C$ of genus $g$ over a finite
field~$\Fq$.  Initially this increased interest was motivated in part
by new constructions of error-correcting codes exceeding the 
Gilbert--Varshamov bound, but now there are many problems related to the
computation of $N_q(g)$ that are mathematically attractive in their own
right, independent of possible applications in coding theory.

In the 1940s, Andr\'e Weil~\cites{Weil1940,Weil1941,Weil1945,Weil1946}
showed that if $C$ is a genus-$g$ curve over $\Fq$, then
\[ q+1-2g\sqrt{q} \le \#C(\Fq) \le q+1+2g\sqrt{q},\]
so that $N_q(g) \le q+1+2g\sqrt{q}.$  In the 1980s this upper bound was
improved in a number of ways.  Serre~\cite{Serre1983a} showed that
\[ N_q(g) \le q + 1 + g \lfloor 2\sqrt{q}\rfloor, \]
and Manin~\cite{Manin} and Ihara~\cite{Ihara} showed that the Weil
bound could be improved even further when $g$ is large with respect 
to~$q$.  Generalizing these ideas, \Drinfeld\ and \Vladut\ showed that
for fixed $q$,
\[
   N_q(g)\leq (\sqrt{q}-1+o(1)) g \qquad\text{as $g \rightarrow \infty$,}
\]
and Serre~\cite{Serre:notes} developed the ``explicit formul\ae''
method (optimized by Oesterl\'e), which gives the best bound on 
$N_q(g)$ that can be obtained formally using only Weil's 
``Riemann hypothesis'' for curves and the fact that for every $d\ge0$
the number of degree-$d$ places on a curve is non-negative. For general
$q$ and $g$ the Oesterl\'e bound has not been improved upon, but for 
certain families and special cases improvements can be made
\cites{FuhrmannTorres, Howe2011, HoweLauter2003, Kohnlein, 
       KorchmarosTorres, Lauter1999, Lauter2000, Lauter:conf, 
       LauterSerre2001, LauterSerre2002, Savitt, Serre1983a, 
       Serre:notes, Serre1983b, Serre1984, Stark, StohrVoloch,
       Zaytsev}.

In 2000, van der Geer and van der Vlugt~\cite{GeerVlugt2000} published
a table of the best upper and lower bounds on $N_q(g)$ known at the
time, for $g\le 50$ and for $q$ ranging over small powers of $2$ 
and~$3$. They updated their paper twice a year after its publication,
and the revised versions were made available on van der Geer's website.
In 2010, van der Geer, Ritzenthaler, and the authors, with technical
assistance from Geerit Oomens, incorporated the updated tables
from~\cite{GeerVlugt2000} into the online tables now available
at \url{manypoints.org}.  These new online tables display results for
many more prime powers $q$ than were in~\cite{GeerVlugt2000}: the 
primes less than~$100$, the prime powers $p^i$ for $p<20$ and $i\le 5$,
and the powers of $2$ up to~$2^7$.  The original tables of van der Geer
and van der Vlugt inspired us to do the work that appeared in our
earlier paper~\cite{HoweLauter2003}; afterwards, we continued to work
on the problem of improving the known upper bounds on $N_q(g)$, and the
work we present in this paper was used to help populate the 
\texttt{manypoints} tables when the site was created.

In our 2003 paper we used a number of techniques to show that certain
isogeny classes of abelian varieties over finite fields do not contain 
Jacobians of curves; by enumerating the isogeny classes of a given
dimension $g$ over a given field $\Fq$ that could possibly contain a
Jacobian of a curve with $N$ points, and then applying our techniques,
we were able to show that some values of $N$ could not occur.  We were
thus able to improve the known upper bounds on $N_q(g)$ for many 
pairs $(q,g)$.  In this paper, we introduce four new techniques that 
can sometimes be used to show that an isogeny class of abelian 
varieties does not contain a Jacobian.  These new techniques were 
responsible for improving more than $16\%$ of the upper bounds in
the 2009 version of~\cite{GeerVlugt2000} when those results were
integrated into the \texttt{manypoints} tables.

The first of our new techniques concerns isogeny classes containing
product varieties. In our earlier paper, we showed that one can
sometimes deduce arithmetic and geometric properties of curves whose
Jacobians are isogenous to a product $A\times B$ when the resultant
of the radicals of the ``real Weil polynomials'' 
(see Section~\ref{S:gluing}) of $A$ and $B$ is small.  The first 
improvement we introduce here is to show that we can replace the 
resultant by the \emph{reduced resultant} in these arguments. The 
reduced resultant is defined, and the new results are explained, in 
Section~\ref{S:gluing}.  We also explain how in certain circumstances
we can replace the reduced resultant by an even smaller quantity that
depends more delicately on the varieties $A$ and~$B$.

In our earlier paper, we showed that if $E$ is a supersingular 
elliptic curve over $\Fq$ with all endomorphisms defined over~$\Fq$,
and if $A$ is an ordinary elliptic curve such that the resultant of
the real Weil polynomials of $E$ and $A$ is squarefree, then there
is no Jacobian isogenous to $E^n\times A$ for any $n>0$.  Our second
improvement is to show that the same statement holds when $A$ is an
arbitrary ordinary variety.  This is explained in 
Section~\ref{S:supersingular}.

Our third new technique concerns isogeny classes that contain a variety
of the form $A\times E^n$, where $E$ is an ordinary elliptic curve 
over~$\Fq$.  We show in Section~\ref{S:Hermitian} that if a curve $C$
has Jacobian isogenous to $A\times E^n$, then there is a map from $C$ 
to $E$ whose degree is bounded above by an explicit function of the
discriminant of $\End E$, the reduced resultant of the real Weil
polynomials of $E$ and~$A$, and the exponent $n$.  In order to produce
the sharpest bounds possible, we give an algorithm for computing the 
length of the shortest nonzero vectors in small Hermitian lattices over
imaginary quadratic fields of class number~$1$.  We provide tables of
some of these sharp upper bounds in Section~\ref{S:Hermitian}.

Our fourth technique is a theorem that gives an easy-to-check necessary
and sufficient condition for the entire category of abelian varieties
in a given ordinary isogeny class over a finite field to be definable
over a subfield.  We present this result and explain its significance
in Section~\ref{S:Galois}.

We have implemented all of our calculations in the computer algebra
package Magma~\cite{magma}. The programs we use are found in the 
package \texttt{IsogenyClasses.magma}, which is available on the first
author's website: Go to the bibliography page

\centerline{\href{http://alumni.caltech.edu/~however/biblio.html}
                {\texttt{http://alumni.caltech.edu/{\lowtilde}however/biblio.html}}
}

\noindent
and follow the link associated to this paper.  We outline the structure
of these Magma routines in Section~\ref{S:Magma}, and in 
Section~\ref{S:results} we present a sampling of the computational 
results we have obtained.  These include two new values of $N_q(g)$ and
an analysis of the possible Weil polynomials of genus-$12$ curves over
$\FF_2$ meeting the Oesterl\'e bound.

As we have mentioned, some of our arguments give upper bounds for the
degrees of maps from a curve $C$ to an elliptic curve~$E$.  Such upper
bounds give restrictions on the determinant of the Mordell--Weil 
lattice of the base extension of $E$ to the function field $K$ of~$C$.
In Section~\ref{S:Sha} we indicate how some of our results, when 
combined with proven cases of the conjecture of Birch and 
Swinnerton--Dyer, allow us to give lower bounds (and sometimes even
exact formulas) for the size of the Shafarevich--Tate group of $E$ 
over~$K$.

\section{Reduced resultants and the gluing exponent}
\label{S:gluing}

In our previous paper~\cite{HoweLauter2003} we analyzed non-simple
isogeny classes of abelian varieties over finite fields by bounding the
`distance' between a variety $A$ in the isogeny class and a product 
variety, measured essentially by the degree of the smallest isogeny
from $A$ to a product. We continue to use this same strategy, but we 
will improve upon our earlier bounds.

In the following definition, we use the convention that the greatest 
common divisor of the set $\{ 0\}$ is~$\infty$.

\begin{definition}
\label{D:GE}
Let $A_1$ and $A_2$ be abelian varieties over a finite field $k$. Let 
$E$ be the set of integers $e$ with the following property\textup{:}
If $\Delta$ is a finite group scheme over $k$ that can be embedded in
a variety isogenous to $A_1$ and in a variety isogenous to~$A_2$, then
$e\Delta = 0$.  We define the \emph{gluing exponent} $e(A_1,A_2)$ of 
$A_1$ and $A_2$ to be the greatest common divisor of the set~$E$.
\end{definition}

If $A_1$ and $A_2$ have no isogeny factor in common, there exist
\emph{nonzero} integers $e$ with the property mentioned in the
definition --- for example, the proof of 
\cite{HoweLauter2003}*{Lem.~7, p.~1684} shows that the quantity 
$s(A_1,A_2)$ defined in \cite{HoweLauter2003}*{\S1} has the desired 
property --- so $e(A_1,A_2)$ is finite in this case.  Clearly 
$e(A_1,A_2)$ depends only on the isogeny classes of $A_1$ and~$A_2$.

We make this definition because many of the results 
in~\cite{HoweLauter2003} remain true if their statements are modified 
by replacing $s(A_1,A_2)$ with $e(A_1,A_2)$; this is so because the 
only property of $s(A_1,A_2)$ used in the proofs of these results is 
that it lies in the set $E$.  In particular, 
\cite{HoweLauter2003}*{Thm.~1, p.~1678} becomes the following.

\begin{theorem}
\label{T:GE} 
Let $A_1$ and $A_2$ be nonzero abelian varieties over a finite 
field~$k$.
\begin{itemize} 
\item[(a)] If $e(A_1,A_2) = 1$ then there is no curve $C$ over $k$ 
           whose Jacobian is isogenous to $A_1\times A_2$. 
\item[(b)] Suppose $e(A_1,A_2) = 2$.  If $C$ is a curve over $k$ 
           whose Jacobian is isogenous to $A_1\times A_2$, then there
           is a degree-$2$ map from $C$ to another curve $D$ over $k$
           whose Jacobian is isogenous to either $A_1$ or~$A_2$. \qed
\end{itemize} 
\end{theorem}

Also, \cite{HoweLauter2003}*{Lem.~7, p.~1684} becomes:

\begin{lemma}
\label{L:GE}
Suppose $B$ is an abelian variety over a finite field $k$ isogenous to
a product $A_1\times A_2$, where $e(A_1,A_2) < \infty$.  Then there 
exist abelian varieties $A_1'$ and $A_2'$, isogenous to $A_1$ 
and~$A_2$, respectively, and an exact sequence 
\[ 0\to \Delta \to A_1'\times A_2' \to B \to 0 \]
such that the projection maps $A_1'\times A_2' \to A_1'$ and 
$A_1'\times A_2' \to A_2'$ give monomorphisms from $\Delta$ to 
$A_1'[e]$ and to $A_2'[e]$, where $e = e(A_1,A_2)$.

Suppose in addition that $B$ has a principal polarization~$\mu$. Then
the pullback of $\mu$ to $A_1' \times A_2'$ is a product polarization
$\lambda_1 \times \lambda_2$, and the projection maps
$A_1'\times A_2' \to A_1'$ and $A_1'\times A_2' \to A_2'$ give
isomorphisms of $\Delta$ with $\ker \lambda_1$ and $\ker \lambda_2$.
In particular, $\Delta$ is isomorphic to its own Cartier dual.
\qed
\end{lemma}

Furthermore, \cite{HoweLauter2003}*{Prop.~11, p.~1688} becomes:

\begin{proposition}
\label{P:general}
Let $A_1$ and $A_2$ be abelian varieties over a finite field $k$, and
let $e = e(A_1,A_2)$. Suppose that $e<\infty$ and that for every $A_1'$ 
isogenous to $A_1$ and every $A_2'$ isogenous to $A_2$, the only
self-dual finite group-scheme that can be embedded in both $A_1'[e]$ 
and $A_2'[e]$ as the kernel of a polarization is the trivial
group-scheme. Then there is no curve over $k$ with Jacobian isogenous
to $A_1\times A_2$. 
\qed
\end{proposition}

And finally, \cite{HoweLauter2003}*{Prop.~13, p.~1689} becomes:

\begin{proposition}
\label{P:ec}
Suppose $C$ is a curve over a finite field $k$ whose Jacobian is 
isogenous to the product $A\times E$ of an abelian variety $A$ with an
elliptic curve $E$, where $e(A,E) < \infty$.  Then there is an elliptic
curve $E'$ isogenous to $E$ for which there is map from $C$ to $E'$ of
degree dividing $e(A,E)$, and we have $\#C(k) \le e(A,E) \cdot \#E(k)$.
\qed
\end{proposition}

These results will only be more useful than their predecessors 
from~\cite{HoweLauter2003} if we can produce bounds on $e(A_1,A_2)$ 
that are better than the bound $e(A_1,A_2)\mid s(A_1,A_2)$ that 
follows from the proof of~\cite{HoweLauter2003}*{Lem.~7, p.~1684}. The
main result of this section, Proposition~\ref{P:RR} below, provides
such improved bounds.  To state the proposition, we must introduce the
idea of the \emph{reduced resultant} of two polynomials, together with
a result about its computation.

\begin{definition}[Pohst \cite{Pohst}*{p.~179}]
\label{D:RR}
The \emph{reduced resultant} of two polynomials $f$ and $g$ in $\ZZ[x]$
is the non-negative generator of the ideal $\ZZ\cap(f,g)$ of~$\ZZ$,
where $(f,g)$ is the ideal of $\ZZ[x]$ generated by $f$ and~$g$.
\end{definition}

Alternatively (but equivalently), one can define the reduced resultant 
of $f$ and $g$ to be the characteristic of the quotient ring 
$\ZZ[x]/(f,g)$.  The reduced resultant of $f$ and $g$ is $0$ if and
only if $f$ and $g$ are both divisible by a nonconstant polynomial; 
also, the reduced resultant divides the usual resultant, and is 
divisible by all of the prime divisors of the usual resultant.

Computing reduced resultants of monic elements of $\ZZ[x]$ is 
straightforward, as the following lemma shows.

\begin{lemma}
\label{L:computingRR}
Let $f_1$ and $f_2$ be coprime elements of $\ZZ[x]$, not both constant,
whose leading coefficients are coprime to one another, and let $n$ be 
the reduced resultant of $f_1$ and $f_2$.  
\begin{enumerate}
\item There are unique polynomials $b_1$ and $b_2$ in $\ZZ[x]$ such 
      that
\begin{itemize}
\item $\deg b_1 < \deg f_2$
\item $\deg b_2 < \deg f_1$
\item $n = b_1 f_1 + b_2 f_2$.
\end{itemize}
\item Let $a_1$ and $a_2$ be the unique elements of $\QQ[x]$ such that
\begin{itemize}
\item $\deg a_1 < \deg f_2$
\item $\deg a_2 < \deg f_1$
\item $1 = a_1 f_1 + a_2 f_2$.
\end{itemize}
Then the reduced resultant $n$ of $f_1$ and $f_2$ is the least common
multiple of the denominators of the coefficients of $a_1$ and~$a_2$,
and the polynomials $b_1$ and $b_2$ from statement~\textup{(1)} satisfy
$b_1 = na_1$ and $b_2 = na_2$.
\end{enumerate}
\end{lemma}

\begin{proof}
First we prove that there exist polynomials $b_1$ and $b_2$ with the 
properties listed in statement~(1).  The unicity of these polynomials 
will follow from statement~(2).

Let $b_1$ and $b_2$ be arbitrary elements of $\ZZ[x]$ such that 
$n = b_1 f_1 + b_2 f_2$.  We will show that if $\deg b_1 \ge \deg f_2$ 
or $\deg b_2 \ge \deg f_1$ then there are elements $\bbar_1$ and 
$\bbar_2$ of $\ZZ[x]$, whose degrees are smaller than those of $b_1$ 
and $b_2$, such that $n = \bbar_1 f_1 + \bbar_2 f_2$.  By successively
replacing the $b$'s with the $\bbar$'s, we will find a pair $(b_1,b_2)$
of polynomials that satisfy the conditions in statement~(1).

Let $d_1$ and $d_2$ be the degrees of $b_1$ and $b_2$.  Since 
$b_1 f_1 + b_2 f_2$ has degree $0$, we see that 
\[ d_1 + \deg f_1 = d_2 + \deg f_2,\]
so that $d_1 \ge \deg f_2$ if and only if $d_2 \ge \deg f_1$.  For each
$i$ let $u_i$ be the leading coefficient of $b_i$ and let $v_i$ be the
leading coefficient of $f_i$.  Then we must have $u_1v_1 + u_2v_2 = 0$,
and since $v_1$ and $v_2$ are coprime to one another, $v_1$ divides 
$u_2$ and $v_2$ divides $u_1$.  If we set
\begin{align*}
\bbar_1 &= b_1 - (u_1/v_2) x^{d_1 - \deg f_2} f_2 \\
\bbar_2 &= b_2 - (u_2/v_1) x^{d_2 - \deg f_1} f_1 
\end{align*}
then $\bbar_1$ and $\bbar_2$ have the desired properties.

Let $a_1$ and $a_2$ be as in statement~(2).  The unicity of the $a$'s
shows that we must have $b_1 = n a_1$ and $b_2 = n a_2$, and it follows
that $n$ is a multiple of all of the denominators of the coefficients
of $a_1$ and~$a_2$.

On the other hand, if $m$ is the least common multiple of the
denominators of $a_1$ and~$a_2$, then $ma_1$ and $ma_2$ are elements
of $\ZZ[x]$, and the equality $m = (ma_1)f_1 + (ma_2) f_2$ shows that
$m$ is a multiple of $n$.  Therefore $n = m$, and the lemma is proved.
\end{proof}

We are almost ready to state Proposition~\ref{P:RR}.  Recall that the
\emph{Weil polynomial} of a $d$-dimensional abelian variety $A$ over a
finite field $\Fq$ is the characteristic polynomial $f$ of the
Frobenius endomorphism of $A$, and that the \emph{real Weil polynomial}
of $A$ is the unique polynomial $h\in\ZZ[x]$ such that 
$f(x) = x^d h(x + q/x).$  Recall also that the \emph{radical} of a 
polynomial is the product of its irreducible factors, each taken once.
The radical of the real Weil polynomial of an abelian variety over a
finite field is the minimal polynomial of the endomorphism $F+V$, where 
$F$ is Frobenius and $V$ is Verschiebung.

\begin{proposition}
\label{P:RR}
Let $A_1$ and $A_2$ be nonzero abelian varieties over a finite field 
$k$ with no isogeny factor in common.  Let $g_1$ and $g_2$ be the 
radicals of their real Weil polynomials, let $n$ be the reduced 
resultant of $g_1$ and $g_2$, and let $b_1$ and $b_2$ be the unique
elements of $\ZZ[x]$ such that $n = b_1 g_1 + b_2 g_2$ and 
$\deg b_1 < \deg g_2$ and $\deg b_2 < \deg g_1$.
\begin{enumerate}
\item The gluing exponent $e(A_1,A_2)$ divides $n$.
\item Let $g = g_1 g_2$, and suppose that $g$ is divisible by 
      $x^2 - 4q$, where $q = \# k$.  If the coefficients of the 
      polynomial $b_1 g_1 + x g / (x^2 - 4q)$ are all even, then 
      $n$ is even, and $e(A_1,A_2)$ divides $n/2$.
\end{enumerate}
\end{proposition}

We prove Proposition~\ref{P:RR} at the end of this section, after we
state and prove two lemmas.  Throughout the rest of this section, 
$A_1$ and $A_2$ will be abelian varieties as in the statement of the 
proposition, and $F$ and $V$ will be the Frobenius and Verschiebung 
endomorphisms of $A_1\times A_2$.

Our first lemma shows that we can find bounds on $e(A_1,A_2)$ by 
understanding the endomorphism ring of $A_1\times A_2$.

\begin{lemma}
\label{L:rings}
For each $i$ let $\varphi_i$ be the projection map 
$\End(A_1\times A_2)\to \End A_i$.  Suppose $\beta$ is an element of
the subring $\ZZ[F,V]$ of the center of $\End(A_1\times A_2)$ with the
property that $\varphi_1(\beta)=0$ and $\varphi_2(\beta)$ is an 
integer~$n$. Then the gluing exponent $e(A_1,A_2)$ is a divisor of~$n$.
\end{lemma}

\begin{proof}
Suppose $A_1'$ and $A_2'$ are abelian varieties over $k$ that are
isogenous to $A_1$ and~$A_2$, respectively.  Note that $F$ and $V$
are endomorphisms of $A_1'\times A_2'$, and that every isogeny from
$A_1'\times A_2'$ to $A_1\times A_2$ respects the actions of $F$ 
and~$V$.  Therefore $\beta$ can be viewed as an endomorphism of 
$A_1'\times A_2'$, and the projections of $\beta$ to $\End A_1'$ 
and to $\End A_2'$ are $0$ and $n$, respectively.

Suppose $\Delta$ is a finite group scheme over $k$ for which there are 
monomorphisms $\Delta\hookrightarrow A_1'$ and 
$\Delta\hookrightarrow A_2'$ for some $A_1'$ and $A_2'$ isogenous to 
$A_1$ and $A_2$.  Frobenius and Verschiebung also act on~$\Delta$, and 
the existence of a monomorphism from $\Delta$ to $A_1'$ shows that 
$\beta$ acts as $0$ on $\Delta$.  But the existence of a monomorphism
from $\Delta$ to $A_2'$ shows that $\beta$ acts as $n$ on $\Delta$.
Therefore $\Delta$ is killed by $n$, and $e(A_1,A_2)$ is a divisor 
of~$n$.
\end{proof}

\begin{lemma}
\label{L:RR}
Let $\alpha$ be an element of the subring $\ZZ[F,V]$ of
$\End(A_1\times A_2)$ and let $g_1$ and $g_2$ be the minimal 
polynomials of $\alpha$ restricted to $A_1$ and $A_2$, respectively.
Then the gluing exponent $e(A_1,A_2)$ is a divisor of the reduced 
resultant of $g_1$ and~$g_2$.
\end{lemma}

\begin{proof}
Let $n$ be the reduced resultant of $g_1$ and $g_2$, so that there are
elements $b_1$ and $b_2$ of $\ZZ[x]$ such that $n = b_1 g_1 + b_2 g_2$.
Let $\beta = (b_1 g_1)(\alpha)$.  Then $\beta$ acts as $0$ on~$A_1$, 
because it is a multiple of $g_1(\alpha)$, and it acts as $n$ on~$A_2$,
because $n-\beta = (b_2 g_2)(\alpha)$ is a multiple of $g_2(\alpha)$.
It follows from Lemma~\ref{L:rings} that $e(A_1,A_2)$ divides $n$.
\end{proof}

\begin{proof}[Proof of Proposition~\textup{\ref{P:RR}}]
Statement (1) follows from Lemma~\ref{L:RR} and the fact that $g_1$ 
and $g_2$ are the minimal polynomials of $F+V$ restricted to $A_1$ 
and~$A_2$, respectively.

Suppose $g$ is divisible by $x^2 - 4q$, say $g = (x^2 - 4q) h$.  Let
$S$ be the subring $\ZZ[F,V]$ of the center of $\End(A_1\times A_2)$, and let $R$ be 
the subring $\ZZ[F+V]$ of~$S$. The tensor product $R_\QQ = R\otimes\QQ$
is a product of fields~$K_i$, with each $K_i$ corresponding to an 
irreducible factor of $g$, and the tensor product $S_\QQ = S\otimes\QQ$
is a product of fields $L_i$, with each $L_i$ being an extension 
of~$K_i$, of degree $1$ if the corresponding factor of $g$ divides 
$x^2-4q$, and of degree $2$ otherwise.  Let $K$ be the product of the 
$K_i$ corresponding to factors of $g$ that do not divide $x^2 - 4q$, 
let $K'$ be the product of the remaining~$K_i$, and let $L$ and $L'$ be
the products of the corresponding $L_i$.  Note that on each factor 
$L_i$ of $L'$ we have $F = V$.

Suppose the coefficients of the polynomials $r = b_1g_1 + x h$ are 
all even.  Consider the polynomial in $\ZZ[u,v]$ obtained by evaluating
$r$ at $u+v$.  Its coefficients are also all even, so the same is true
of the polynomial
\[b_1(u+v) g_1(u+v) + (u - v) h(u + v).\]
Let $s$ be $1/2$ times this polynomial, so that $s$ lies in $\ZZ[u,v]$,
and $s(F,V)$ lies in $S$.

Consider the element $(F - V) h(F+V)$ of $S_\QQ$.  On each factor $L_i$
of~$L$, this element is $0$, because $h(F+V) = 0$ on each factor $K_i$
of~$K$.  But on each factor $L_i$ of $L'$, we have $F - V = 0$.  Thus,
$(F - V) h(F+V) = 0$.

We see that $2s(F,V)$ differs from $b_1(F+V) g_1(F+V)$ by~$0$, so 
$2s(F,V)$ is equal to $0$ in $\End A_1$ and is equal to $n$ 
in~$\End A_2$.  Since $s(F,V)$ lies in $\End(A_1\times A_2)$, it 
follows that $n$ is even, and from Lemma~\ref{L:rings} we see that the
gluing exponent of $A_1$ and $A_2$ divides $n/2$.
\end{proof}

\section{Supersingular factors in the Jacobian}
\label{S:supersingular}

We say that a Weil polynomial or a real Weil polynomial is 
\emph{ordinary} if the corresponding isogeny class consists of ordinary
abelian varieties.  In this section we will prove the following theorem,
which generalizes~\cite{HoweLauter2003}*{Cor.~12, p.~1689}.

\begin{theorem}
\label{T:SSfactor}
Suppose $h\in\ZZ[x]$ is the real Weil polynomial of an isogeny class of
abelian varieties over a finite field $k$, where $q = \#k$ is a square.
Suppose further that $h$ can be written $h = (x - 2s)^n h_0$, where 
$s^2 = q$, where $n>0$, and where $h_0$ is a nonconstant ordinary real
Weil polynomial such that the integer $h_0(2s)$ is squarefree.  Then 
there is no curve over $k$ with real Weil polynomial equal to $h$.
\end{theorem}

The heart of the proof is a lemma about finite group schemes over
finite fields.

\begin{lemma}
\label{L:F=V}
Suppose $\Delta$ is a finite $\ell$-torsion group scheme over a finite
field $k$ with $\#\Delta > \ell$, and suppose the Frobenius and
Verschiebung endomorphisms of $\Delta$ act as multiplication by an
integer $s$ with $s^2 = q$.  If $\Delta$ can be embedded in an ordinary
abelian variety $A$ over~$k$ with real Weil polynomial $h_0$, then the
integer $h_0(2s)$ is divisible by~$\ell^2$.
\end{lemma}

\begin{proof}
First we prove the lemma under the additional assumption that $A$ is 
simple.

Let $R$ be the subring $\ZZ[F,V]$ of $\End A$, where $F$ and $V$ are
the Frobenius and Verschiebung endomorphisms.  Since $A$ is ordinary
and simple, the ring $R$ is an order in a CM-field $K$ whose degree
over $\QQ$ is twice the dimension $d$ of~$A$. If we define $R^+$ to be
the subring $\ZZ[F+V]$ of~$R$, then $R^+$ is an order in the maximal
real subfield $K^+$ of~$K$. (These facts follow, for instance, from 
the Honda--Tate theorem~\cite{Tate1968}*{Th\'eor\`eme~1, p.~96}.)  It is
easy to see that the elements
\[ 1, \ F, \ V, \ F^2, \ V^2, \ \ldots, \ F^{d-1}, \ V^{d-1}, \ F^d \]
form a basis for $R$ as a $\ZZ$-module.

Let $\gothp$ be the ideal $(\ell, F - s, V - s)$ of $R$.  We will show
that $\gothp$ is a prime ideal, and that if $h_0(2s)$ is not divisible
by $\ell^2$ then the localization $R_\gothp$ is a regular local ring. 
We will accomplish this by analyzing the rings $R/\gothp$ and 
$R/\gothp^2$.

Let $F'$ and $V'$ be the Frobenius and Verschiebung endomorphisms of
$\Delta$, and choose an embedding $\iota$ of $\Delta$ into~$A$.  The
embedding $\iota$ gives us a homomorphism from $R$ to $\End \Delta$ 
that sends $F$ to $F' = s$ and $V$ to $V' = s$, and clearly $\gothp$ is
contained in the kernel of this homomorphism. Therefore $\gothp$ is not
the unit ideal.  (We note that the fact that $\gothp$ is not the unit
ideal tells us that $s$ is coprime to $\ell$; for otherwise $\gothp$ 
would contain $F$ and $V$, and $F$ and $V$ are coprime in $R$ because
$A$ is ordinary --- see~\cite{Howe1995}*{Lem.~4.12, p.~2372}.)

On the other hand, since $F\equiv V \equiv s \bmod \gothp$, we see that
every power of $F$ and $V$ is congruent to an integer modulo $\gothp$,
so every element of $R$ is congruent to an integer modulo $\gothp$.
Furthermore, since $\ell\in\gothp$, we find that every element of $R$
is congruent modulo $\gothp$ to a nonnegative integer less than~$\ell$.
This shows that $R/\gothp \cong \FF_\ell$, so $\gothp$ is prime.

Now we analyze the ring $R/\gothp^2$. Note that $\gothp^2$ contains
$(F-s)^2 = F^2 - 2sF + q$ and $(V-s)^2 = V^2 - 2sV + q$.  Using 
multiples of these elements to eliminate higher powers of $F$ and~$V$,
we see that every element of $R/\gothp^2$ can be represented by an 
element of the form $a F + b V + c$.  Using the fact that $\gothp^2$
contains $\ell^2$, $\ell(F-s)$, and $\ell(V-s)$, we see that every 
element of $R/\gothp^2$ can be represented by an element of this form
with the further restriction that $a, b$, and $c$ are nonnegative 
integers with $a,b<\ell$ and $c<\ell^2$.

Now, $\gothp^2$ also contains $(F-s)(V-s) = 2q - (F+v)s = s(2s-(F+V))$,
and since $s$ is coprime to $\ell$ and hence not in $\gothp$, we see 
that $\gothp^2$ contains $F+V-2s$.  This shows that we can find 
representatives as above for which $b=0$.

Suppose, to obtain a contradiction, that $h_0(2s)$ is not divisible 
by~$\ell^2$.  Since $F+V \equiv 2s \bmod \gothp^2$, we have 
$0 = h_0(F+V) \equiv h_0(2s) \bmod \gothp^2$.  This can only happen if
$h_0(2s)$ is divisible by $\ell$, and then our assumption that 
$h_0(2s)$ is not divisible by $\ell^2$ implies that $\gothp^2$ 
contains~$\ell$. This means that every element of $R/\gothp^2$ has a 
representative of the form $aF + c$ where $a$ and $c$ are both 
non-negative integers less than $\ell$.  In particular, 
$\gothp/\gothp^2$ is a $1$-dimensional $R/\gothp$-vector space, and so
$R_\gothp$ is a regular local ring.

Deligne~\cite{Deligne1969} proved a result that implies that there is 
an equivalence of categories between the category of abelian varieties 
over $k$ that are isogenous to $A$ and the category of nonzero 
finitely-generated $R$-submodules of~$K$.  This equivalence of 
categories is fleshed out in~\cite{Howe1995}; the only result we will
use in our argument is that if our variety $A$ corresponds to the 
isomorphism class of an $R$-module $\gothA\subset K$, then the 
$R$-modules $A[\ell]$ and $\gothA/\ell\gothA$ are isomorphic to one 
another. 

The image of $\Delta$ in $A$ must sit inside the largest subgroup of 
$A[\ell]$ on which $F$ and $V$ acts as $s$, so we would like to analyze
the $\gothp$-primary part of $A[\ell]$.  This $\gothp$-primary part is
simply $\gothA_\gothp / \ell \gothA_\gothp$, and since $R_\gothp$ is
regular, this last module is isomorphic to $R_\gothp / \gothp^a$, where
$\gothp^a$ is the largest power of $\gothp$ dividing $\ell$.  The 
submodule of $R_\gothp / \gothp^a$ on which $\gothp$ acts trivially is
$\gothp^{a-1}/\gothp^a$, which has order $\ell$.  Since $\Delta$ has 
order greater than~$\ell$, we see that $\Delta$ cannot be embedded in 
$A[\ell]$, a contradiction.

This proves the lemma in the case where $A$ is simple.  Now we turn to
the general case.  The decomposition of $A$ up to isogeny corresponds
to the factorization of~$h_0$.  Suppose
\[h_0 = h_1^{e_1} h_2^{e_2} \cdots h_r^{e_r}\]
is the factorization of $h_0$ into powers of distinct irreducibles, and
suppose, to obtain a contradiction, that $h_0(2s)$ is not divisible 
by~$\ell^2$. As before, choose an embedding $\iota$ of $\Delta$ 
into~$A$.

The radical of $h_0$ is the minimal polynomial of $F+V$, viewed as an 
element of $\End A$, so $A$ is killed by $h_0(F+V)$.  Since $\Delta$ 
can be embedded in $A$, it is also killed by this polynomial in $F+V$.
But since $F+V = 2s$ on $\Delta$, we find that $h_0(2s)$ 
kills~$\Delta$, and since $\Delta$ is $\ell$-torsion, we see that 
$h_0(2s)$ is divisible by~$\ell$.

This means that exactly one prime factor $h_i$ of $h_0$ has the 
property that $h_i(2s)$ is divisible by $\ell$, and for this $i$ we
must have $e_i = 1$.  By renumbering, we may assume that $h_1(2s)$ 
is divisible by $\ell$ and that $e_1 = 1$.

Let $H = h_2^{e_2} \cdots h_r^{e_r}$, so that $H(2s)$ is coprime
to~$\ell$.  If we apply $H(F+V)$ to $A$, we obtain a subvariety $A_1$ 
of $A$ on which $F+V$ satisfies $h_1$.  Since $H(F+V)$ acts as $H(2s)$
on $\Delta$, we see that the image of $\iota(\Delta)$ under $H(F+V)$
is simply $\iota(\Delta)$.  Thus, $\iota$ provides an embedding of
$\Delta$ into the simple variety~$A_1$. As we have shown, the existence
of this embedding is inconsistent with the fact that $h_1(2s)$ is not 
divisible by $\ell^2$, and this contradiction proves the lemma.
\end{proof}

\begin{proof}[Proof of Theorem~\textup{\ref{T:SSfactor}}]
Let $E$ be an elliptic curve over $k$ with real Weil polynomial equal
to $x - 2s$ and let $A$ be an abelian variety over $k$ with real Weil
polynomial equal to $h_0$.  Let $e$ be the gluing exponent of $E^n$ 
and~$A$.  Proposition~\ref{P:RR} says that $e$ divides the reduced 
resultant of $x-2s$ and the radical of $h_0$; this reduced resultant
divides the resultant of $x-2s$ and $h_0$, which is the squarefree 
integer $h_0(2s)$, so $e$ is squarefree.

Suppose $\Delta$ is a nontrivial self-dual group scheme that can be
embedded in a variety isogenous to $E^n$ as the kernel of a 
polarization, and that can also be embedded a variety isogenous to~$A$
as the kernel of a polarization. Let $\ell$ be a prime divisor of the
order of $\Delta$; then the $\ell$-primary part $\Delta_\ell$ of 
$\Delta$ is a nontrivial group scheme that can be embedded in a 
variety isogenous to $E^n$ and in a variety isogenous to~$A$.  
Furthermore, since $e\Delta_\ell = 0$, and $e$ is squarefree, we see
that $\Delta_\ell$ is $\ell$-torsion.  And finally, since $\Delta$ is
isomorphic to the kernel of a polarization and hence has square order,
we see that the order of $\Delta_\ell$ is divisible by~$\ell^2$.

Frobenius and Verschiebung act on $E$ as multiplication by the
integer~$s$, so they act on $E^n$ and on all varieties isogenous to 
$E^n$ in the same way.  Since $\Delta_\ell$ can be embedded in a 
variety isogenous to $E^n$, we see that Frobenius and Verschiebung act
as multiplication by $s$ on $\Delta_\ell$ as well.

But then $\Delta_\ell$ satisfies the hypotheses of Lemma~\ref{L:F=V},
so we find that $h_0(2s)$ is divisible by $\ell^2$, a contradiction.
Thus, no nontrivial self-dual group scheme can be embedded as the
kernel of a polarization both in a variety isogenous to $E^n$ and in
a variety isogenous to $A$.  It follows from 
Proposition~\ref{P:general} that there is no curve over $k$ with real
Weil polynomial equal to~$h$.
\end{proof}

\section{Hermitian lattices}
\label{S:Hermitian}

We saw in Proposition~\ref{P:ec} that if the Jacobian $J$ of a curve 
$C$ over a finite field is isogenous to a product $A\times E$, where
$A$ is an abelian variety and $E$ is an elliptic curve with 
$\Hom(E,A) = \{0\}$, then we can derive an upper bound on the degree of
the smallest-degree map from $C$ to an elliptic curve isogenous to~$E$.
Our goal in this section is to prove a similar result when $J$ is 
isogenous to $A\times E^n$, for $n>0$ and $E$ ordinary.

\begin{proposition}
\label{P:newec}
Suppose $C$ is a curve over $\Fq$ whose Jacobian is isogenous to a
product $A\times E^n$, where $n>0$, where $E$ is an ordinary elliptic 
curve of trace~$t$, and where $A$ is an abelian variety such that 
the gluing exponent $e = e(A,E)$ is finite.  Let $h$ be the real Weil
polynomial of $A$ and let $b = \gcd(e^n,h(t))$.  Let $E'$ be any
elliptic curve isogenous to $E$ whose endomorphism ring is generated
over $\ZZ$ by the Frobenius.  Then there is a map from $C$ to $E'$ of
degree at most
\[ \gamma_{2n}\, b^{1/n} \sqrt{|t^2 - 4q|/4},\]
where $\gamma_{2n}$ is the Hermite constant for dimension~$2n$.
\end{proposition}

\begin{remark}
We have
\[
    \gamma_2^2 = 4/3,  \quad
    \gamma_4^4 = 4,    \quad 
    \gamma_6^6 = 64/3, \quad 
    \gamma_8^8 = 256,  \quad \text{and\ }
    \gamma_{10}^{10} < 5669.
\]
The value of $\gamma_2$ was given by Hermite~\cite{Hermite1850}, the
value of $\gamma_4$ by Korkine and Zolotareff~\cite{KorkineZolotareff},
and the values of $\gamma_6$ and $\gamma_8$ by 
Blichfeldt~\cite{Blichfeldt1935}. The upper bound for $\gamma_{10}$ 
follows from an estimate of Blichfeldt~\cite{Blichfeldt1914}.
General upper bounds for $\gamma_n$ can be found 
in~\cite{GL1987}*{\S38}.
\end{remark}

In any specific instance, it may be possible to improve the bound from
Proposition~\ref{P:newec} by using refinements of the individual lemmas
from which the proof of the proposition is built.  We turn now to these
lemmas.  After presenting the proof of Proposition~\ref{P:newec}, we
will explore some cases in which the proposition can be improved.

\begin{lemma}
\label{L:quadform}
Let $E$ be an ordinary elliptic curve over $\Fq$, let $R$ be the 
endomorphism ring of $E$, let $A$ be an abelian variety isogenous to
$E^n$ such that $R$ is contained in the center of~$\End A$, and let 
$\Lambda$ be a polarization of~$A$.  Let $Q\col \Hom(E,A)\to \ZZ$ be
the map that sends $0$ to $0$ and that sends a nonzero morphism $\psi$
to the square root of the degree of the pullback polarization 
$\psi^*\Lambda$. Then $Q$ is a positive definite quadratic form on 
$\Hom(E,A)$, the determinant of $Q$ is equal to
$\left|(\disc R)/4\right|^n \deg\Lambda$, and there is a nonzero
element $\psi$ of $\Hom(E,A)$ such that
\[Q(\psi) \le \gamma_{2n}\, (\deg\Lambda)^{1/(2n)} \sqrt{\left|\disc R\right|/4}.\]
\end{lemma}

\begin{remark}
By the \emph{determinant} of a positive definite quadratic form 
$Q\col L\to \ZZ$ on a free $\ZZ$-module~$L$, we mean the following: Let
$B$ be the unique symmetric bilinear form $L\times L\to\QQ$ such that
$Q(x) = B(x,x)$ for all $x$. Then we define the determinant of $Q$ to
be the determinant of the Gram matrix for~$B$.

If we let $L^* = \Hom(L,\ZZ)$ be the dual of~$L$, then $2B$ defines a
homomorphism $b\col L\to L^*$, and we have 
\[ \det Q = (1/2)^{\rank L} [L^*\col b(L)] .\]
\end{remark}

\begin{proof}[Proof of Lemma~\textup{\ref{L:quadform}}]
Let $p$ be the characteristic of $\Fq$, let $\pi$ and $\pibar$ be the
Frobenius and Verschiebung endomorphisms of~$E$, and let $R_0$ be the
subring $\ZZ[\pi,\pibar]$ of~$R$.  The theory of Deligne 
modules~\cites{Deligne1969,Howe1995} shows that the category of abelian
varieties over $\Fq$ that are isogenous to some power of~$E$ is
equivalent to the category of torsion-free finitely generated 
$R_0$-modules. The equivalence depends on a choice: we must specify an
embedding $\eps\col W\to\CC$ of the Witt vectors $W$ over $\Fqbar$ into
the complex numbers.  The equivalence sends an abelian variety to the
first integral homology group of the complex abelian variety obtained
by base-extension (via $\eps$) from the canonical lift of the variety.
It follows from the irreducibility of Hilbert class polynomials that we
can choose the embedding $\eps$ so that the equivalence takes the 
elliptic curve $E$ to the $R_0$-module $R$.

We recall from~\cite{Howe1995}*{\S4} how the concept of a polarization
translates to the category of Deligne modules, at least in the special
case we are considering. Let $K$ be the quotient field of $R_0$. The
embedding $\eps$ determines a $p$-adic valuation $\nu$ on the field of
algebraic numbers sitting inside the complex numbers; we let 
$\varphi\col K\to\CC$ be the complex embedding of $K$ such that 
$\nu(\varphi(\pi))>0$.  A \emph{polarization} on a finitely-generated
torsion-free $R_0$-module $M$ is a skew-Hermitian form 
\[S\col (M\otimes\QQ) \times (M\otimes\QQ) \to K\]
such that $\Tr_{K/\QQ}S(M,M)\subseteq\ZZ$ and such that 
$\varphi(S(x,x))$ lies in the lower half-plane for all nonzero 
$x\in M$.  (Note that $\varphi(S(x,x))$ must be pure imaginary, since
$S$ is skew-Hermitian.)
The composition $\Tr_{K/\QQ}\circ S$ gives a map $M\to\Hom(M,\ZZ)$; 
the \emph{degree} of the polarization is the size of the cokernel of 
this map.

Now let $M$ be the $R_0$-module corresponding to the $n$-dimensional 
variety $A$.  Since $R$ lies in the center of the endomorphism ring 
of~$M$, we see that the $R_0$-module structure of $M$ extends to an 
$R$-module structure.  Thus, every element $x$ of $M$ determines a map
$\alpha_x\col R\to M$ defined by $\alpha_x(r) = rx$, and every map from
$R$ to $M$ is of this form.  We find that $\Hom(E,A) \cong M$.

Let $S$ be the polarization on $M$ corresponding to the polarization
$\Lambda$ of $A$, and let $\psi$ be a nonzero map from $E$ to $A$,
corresponding to a map $\alpha\col R\to M$, say $\alpha(r) = rx$ for a
nonzero $x\in M$.  The polarization $\varphi^*\Lambda$ then corresponds
to the skew-Hermitian form
\[ S_x\col K\times K\to K\]
defined by
\[S_x(u,v) = S(ux, vx) = u\bar{v} S(x,x).\] 
Our map $R\to\Hom(R,Z)$ is then
\[ v \mapsto \left(u\mapsto \Tr_{K/\QQ} u\bar{v} S(x,x)\right), \]
and the size of the cokernel of this map is the index of the $R$-module
$S(x,x) R$ inside the trace dual of~$R$.  Let $\delta$ be a generator
of the different of $R$, chosen so that $\varphi(\delta)$ is pure
imaginary and in the upper half plane.  Then the size of the cokernel
is the norm of $\delta S(x,x)$.  Since $S(x,x)$ and $\delta$ are both
pure imaginary, and their images under $\varphi$ lie in opposite 
half-planes, the product $\delta S(x,x)$ is a positive rational number,
so its norm is just its square. Thus, under the identification 
$\Hom(E,A)\cong M$, the function $Q$ in the statement of the lemma is
the map $M\to\ZZ$ defined by $Q(x) = \delta S(x,x).$  Therefore $Q$
is a quadratic form.

We compute that the symmetric bilinear form $B$ on $M$ such that 
$Q(x) = B(x,x)$ is given by
\[ B(x,y) = (1/2) \Tr_{K/\QQ} \delta S(x,y) .\]
The map $M\to\Hom(M,\ZZ)$ determined by $\Tr_{K/\QQ}\circ S$ has 
cokernel of size $\deg\Lambda$; replacing $S$ with $\delta S$ increases
the size of the cokernel by $\Norm(\delta)^n = \left|\disc R\right|^n$.
Therefore
\begin{align*}
\det Q  &= (1/2)^{\rank_\ZZ M} \left|\disc R\right|^n \deg\Lambda \\
        &=  \left|(\disc R)/4\right|^n \deg\Lambda.
\end{align*}        
The final statement of the lemma follows 
from~\cite{GL1987}*{Thm.~38.1, p.~386} 
or~\cite{Cassels}*{Thm.~12.2.1, p.~260}.
\end{proof}

\begin{lemma}
\label{L:CtoE}
Let $C$ be a curve over a field $k$, let $E$ be an elliptic curve 
over~$k$, and let $\mu\col E\to \hat{E}$ and 
$\lambda\col \Jac C \to \hat{\Jac C}$ be the canonical polarizations 
of $E$ and of the Jacobian of $C$, respectively. Suppose $C$ has a 
$k$-rational divisor $D$ of degree~$1$, and let $\eps\col C\to \Jac C$
be the embedding that sends a point $P$ to the class of the divisor 
$P - D$.

Suppose $\psi$ is a nonzero homomorphism from $E$ to $\Jac C$, so that
$\hat{\psi}\lambda\psi = d\mu$ for some positive integer~$d$.  Let 
$\varphi\col C \to E$ be the map $\mu^{-1}\hat{\psi}\lambda\eps$.  
Then $\psi = \varphi^*$, and $\varphi$ has degree $d$.
\end{lemma}

\begin{proof}
The maps in the lemma can be arranged into the following diagram:
\[
\xymatrix{
   & E\ar[r]^{d}\ar[d]_{\psi}       & E\ar[r]^{\mu}_{\sim}               &\hat{E} \\
C\ar[r]^-{\eps}\ar`d[r]`[rrr]+/r 3em/`[u]+/u 2em/_{\varphi}`[rr][rru]
   & \Jac C\ar[rr]^{\lambda}_{\sim} &                                    & \hat{\Jac C}\ar[u]_{\hat{\psi}}
}
\]
Inserting another copy of $\Jac C$ into the bottom row we obtain
\[
\xymatrix{
   & E\ar[r]^{d}\ar[d]_{\psi}       & E\ar[r]^{\mu}_{\sim}               &\hat{E} \\
C\ar[r]^-{\eps}\ar`d[r]`[rrr]+/r 3em/`[u]+/u 2em/_{\varphi}`[rr][rru]
   & \Jac C\ar[r]^{1}               &\Jac C\ar[r]^{\lambda}_{\sim}\ar[u] & \hat{\Jac C}\ar[u]_{\hat{\psi}}
}
\]
and we see that the middle vertical map from $\Jac C$ to $E$ is equal
to $\varphi_*$.  Lemma~\ref{L:CtoD}, below, shows that then we must
have $\psi = \varphi^*$, which is the first part of the conclusion of 
the lemma.  Since $\varphi_*\varphi^*$ is equal to multiplication by
the degree of $\varphi$, we find that $\deg\varphi = d$.
\end{proof}

\begin{lemma}
\label{L:CtoD}
Let $f\col C\to D$ be a nonconstant morphism of curves over a 
field~$k$, and let $f_*\col \Jac C\to \Jac D$ and 
$f^*\col \Jac D \to \Jac C$ be the associated push-forward and
pullback maps between the Jacobians of $C$ and $D$.  Under the natural
isomorphisms between $\Jac C$ and $\Jac D$ and their dual varieties,
the isogenies $f_*$ and $f^*$ are dual to one another.
\end{lemma}

\begin{remark}
This statement is proven by Mumford~\cite{Mumford1974b}*{\S1}. Mumford
assumes that $f$ has degree $2$ because his paper is concerned with
double covers; however, the proof does not use this assumption. For the
convenience of the reader, we include a version of his proof here.
\end{remark}

\begin{proof}[Proof of Lemma~\textup{\ref{L:CtoD}}]
Let $\lambda_C$ and $\lambda_D$ be the canonical principal 
polarizations of $\Jac C$ and $\Jac D$.  It suffices to prove the lemma
in the case where $k$ is algebraically closed, so we may assume that $C$
has a $k$-rational point~$P$.  Let $g_C$ be the embedding of $C$ into
$\Jac C$ that sends a point $Q$ to the class of the divisor $Q - P$, 
and let $g_D$ be the embedding of $D$ into $\Jac D$ that sends a point
$Q$ to the class of the divisor $Q - f(P)$.  Then we have a commutative
diagram
\[
\xymatrix{
C\ar[r]^-{g_C}\ar[d]_{f} & \Jac C\ar[d]^{f_*} \\
D\ar[r]^-{g_D}           & \Jac D.
}
\]
Applying the functor $\Pic^0$, we obtain the diagram
\[
\xymatrix{
\Jac C             & \hat{\Jac C}\ar[l]_{\lambda_C^{-1}} \\
\Jac D\ar[u]^{f^*} & \hat{\Jac D}\ar[l]_{\lambda_D^{-1}}\ar[u]_{\hat{f_*}},
}
\]
which expresses the fact that $f^*$ is isomorphic to the dual of $f_*$;
that is, $f^*$ and $f_*$ are dual to one another.
\end{proof}

\begin{proof}[Proof of Proposition~\textup{\ref{P:newec}}]
Suppose $C$ is a curve over a finite field $\Fq$ whose Jacobian is 
isogenous to a product $A\times E^n$, where $n>0$, where $E$ is an 
ordinary elliptic curve with trace $t$, and where $A$ is an abelian
variety such that the gluing exponent $e = E(A,E)$ is finite. Let $E'$ 
be an ordinary elliptic curve over $\Fq$ isogenous to $E$ such that
the endomorphism ring $R$ of $E'$ is generated by the Frobenius; this
means that the discriminant of $R$ is equal to $t^2 - 4q$.

Lemma~\ref{L:GE} says that there is a variety $A'$ isogenous to~$A$, a
variety $B$ isogenous to $E^n$, and an exact sequence
\[ 0 \to \Delta \to A' \times B \to \Jac C \to 0\]
where $\Delta$ is a finite group scheme and the induced maps 
$\Delta\to A'$ and $\Delta\to B$ are monomorphisms.  Let $\lambda$ be
the canonical principal polarization of $\Jac C$.  Again by
Lemma~\ref{L:GE}, pulling $\lambda$ back to $B$ gives us a polarization
$\Lambda$ of $B$ with kernel isomorphic to $\Delta$.  The lemma also
says that $\Delta$ can be embedded into the $e$-torsion of $B$, so the
order of $\Delta$ is a divisor of~$e^{2n}$. 

Let $\eta_\Delta \in \End \Delta$ be the sum of the Frobenius and 
Verschiebung endomorphisms of~$\Delta$, and let $\eta\in\End A'$ be the 
sum of the Frobenius and Verschiebung endomorphisms of~$A'$.  Since
$\Delta$ embeds into $B$, we have $\eta_\Delta = t$ on $\Delta$. 
Thus, the image of $\Delta$ in $A'$ must lie in the kernel of the 
endomorphism $\eta - t$.  The degree of this endomorphism is the 
constant term of its characteristic polynomial, and since the 
characteristic polynomial of $\eta$ is $h^2$, the characteristic
polynomial of $\eta - t$ is $h(x + t)^2$, whose constant term 
is~$h(t)^2$.  Thus, $\Delta$ embeds into a group scheme of 
order~$h(t)^2$, so the order of $\Delta$ is a divisor of $b^2$, where 
$b = \gcd(e^n,h(t))$.

Let $Q$ be the map that sends a nonzero homomorphism $\psi\col E'\to B$
to the square root of the degree of the pullback polarization 
$\psi^*\Lambda$.  Lemma~\ref{L:quadform} says that there is a nonzero
element $\psi$ of $\Hom(E',B)$ such that 
\[ Q(\psi) \le \gamma_{2n}\, (\#\Delta)^{1/(2n)} \sqrt{\left|\disc R\right|/4} 
           \le \gamma_{2n}\, b^{1/n} \sqrt{\left|\disc R\right|/4}.\]
Thus we have a diagram
\[\xymatrix{
E\ar[d]\ar[r]^{\psi^*\Lambda}    & \hat{E} \\
B\ar[d]\ar[r]^{\Lambda}          & \hat{B}\ar[u] \\
\Jac C\ar[r]^\lambda             & \hat{\Jac C}\ar[u]
}
\]
where the vertical arrows on the right are the dual morphisms of the 
vertical arrows on the left.  Using Lemma~\ref{L:CtoE} we obtain a
map from $C$ to $E$ whose degree $d$ is equal to $Q(\psi)$, so that
$d\le \gamma_{2n} \, b^{1/n} \sqrt{|t^2-4q|/4}.$
\end{proof}

As we mentioned earlier, the bound in Proposition~\ref{P:newec} can
sometimes be improved.  There are two places in the proof of the
proposition where improvements can be made:  First, when one has 
specific varieties $A$ and $E$ in hand, the estimate for the size of 
the group scheme $\Delta$ can often be sharpened by a more thorough
analysis of the $e(A,E)$-torsion group schemes that can be embedded 
in a variety isogenous to $E^n$ and in a variety isogenous  to~$A$.
Second, we obtain upper bounds on short vectors for the quadratic form
$Q$ by using general bounds on short vectors in lattices.  But the 
lattices we are considering are quite special --- they come provided
with an action of an imaginary quadratic order --- so there is no
reason to suspect that the bounds for general lattices will be sharp
in our situation.  Improving upper bounds on the lengths of short 
vectors in such lattices is helpful enough in practice that we will 
devote the remainder of this section to doing so.

We will start by studying pullbacks of polarizations on powers of 
elliptic curves.  Then we will focus on the very special case of 
ordinary elliptic curves over finite fields that are isogenous to no
other curves.

\begin{lemma}
\label{L:Enpols}
Let $E$ be an elliptic curve over a field $k$ and let $R = \End E$, so
that $R$ is a ring with a positive involution. Let $\lambda_0$ be the 
canonical principal polarization of $E$.  Fix an integer $n>0$, and let
$\Lambda_0$ denote the product polarization $\lambda_0^n$ of $E^n$.  
Let $\Phi_n$ denote the map from polarizations of $E^n$ to 
$\End E^n \cong M_n(R)$ that sends a polarization $\Lambda$ of $E^n$ 
to the element $\Lambda_0^{-1}\Lambda$ of $\End E^n$. Then\textup{:}
\begin{enumerate}
\item The image of $\Phi_n$ is the set of positive definite Hermitian
      matrices in $M_n(R)$.
\item The degree of a polarization $\Lambda$ of $E^n$ is the square of
      the determinant of $\Phi(\Lambda)$.
\item Let $\Lambda$ be a polarization of $E^n$ and let $H$ be the 
      Hermitian form on $R^n$ determined by $\Phi_n(\Lambda)$.  If 
      $\alpha \col E \to E^n$ is a nonzero map corresponding to a 
      vector $v$ of elements of $R$, then the polarization 
      $\alpha^*\Lambda$ of $E$ is equal to  $d\lambda_0$, where 
      $d = H(v,v)$.
\end{enumerate}
\end{lemma}
      
\begin{proof}
In general, if $A$ is an abelian variety with a principal polarization
$\mu_0$, the map $\mu\to\mu_0^{-1}\mu$ identifies the set of 
polarizations of $A$ with the set of elements of $\End A$ that are 
fixed by the Rosati involution associated to $\mu_0$ and whose minimal
polynomials have only positive real roots. (See the final paragraph of 
\S21 of~\cite{Mumford1974a}.)  The Rosati involution on $\End E^n$
associated to the product polarization $\Lambda_0$ is the conjugate 
transpose, and the roots of the minimal polynomial of a Hermitian 
matrix are all positive precisely when the matrix is positive definite.
This proves~(1).

The degree of an element of $\End E^n$ is equal to the norm (from 
$\End E$ to $\ZZ$) of its determinant.  Since the determinant of a
Hermitian matrix already lies in $\ZZ$, its norm is just its square.
This proves~(2).

Item (3) follows from noting that upon identifying $E$ with its dual
via $\lambda_0$, the dual map $\hat{\alpha}\col E^n\to E$ is given by
the conjugate transpose $v^*$ of the vector $v$.  The pullback of 
$\Lambda$ to $E$ is then given by $v^* \Phi_n(\Lambda) v\lambda_0$,
and this is $H(v,v)\lambda_0$.
\end{proof}

Given an imaginary quadratic order $R$, an integer $n>0$, and an 
integer $D > 0$, let $d(R,n,D)$ be the smallest integer $d$ with the
following property: For every positive definite Hermitian matrix 
$M\in M_n(R)$ of determinant~$D$, the associated Hermitian form over
$R^n$ has a short vector of length at most $d$.  The next lemma shows
that in a very special case, the function $d(R,n,D)$ gives a bound on 
the minimum nonzero value of the function $Q$ from 
Lemma~\ref{L:quadform}.

\begin{lemma}
\label{L:Enbound}
Let $E$ be an elliptic curve over~$\Fq$, let $t$ be the trace of~$E$,
and suppose $t^2 - 4q$ is the discriminant of the maximal order~$R$
of an imaginary quadratic field of class number~$1$.  Let $\Lambda$ be
a polarization of a variety $A$ isogenous to $E^n$, and let 
$Q\col\End(E,A)\to\ZZ$ be as in Lemma~\textup{\ref{L:quadform}}.  Then
there is a nonzero element $\psi\in\End(E,A)$ such that 
$Q(\psi)\le d(R,n,\sqrt{\deg\Lambda}).$
\end{lemma}

\begin{proof}
The theory of Hermitian modules~\cite{LauterSerre2002}*{Appendix}, or 
of Deligne modules~\cites{Deligne1969,Howe1995}, shows that the 
varieties isogenous to $E^n$ correspond to rank-$n$ modules over~$R$.
There is only one such module up to isomorphism, because $R$ has class
number~$1$.  Therefore $A$ is isomorphic to~$E^n$.

Let $\Phi_n$ be as in Lemma~\ref{L:Enpols}, and let $M$ be the
Hermitian matrix $\Phi_n(\Lambda)$, so that part~(2) of the lemma shows
that $\det M = \sqrt{\deg\Lambda}$.   Let $H$ be the Hermitian form on
$R^n$ determined by~$M$. If $\psi\col E\to E^n$ corresponds to a vector
$v\in R^n$, then part~(3) of Lemma~\ref{L:Enpols} shows that 
$Q(\psi) = H(v,v)\le d(R,n,\sqrt{\deg\Lambda}).$
\end{proof}

\begin{lemma}
\label{L:sharp}
The integer entries in Tables~\textup{\ref{T:3}--\ref{T:11}} give 
correct values of $d(R,n,D)$.
\end{lemma}

\begin{proof}
Our proof is computational. We will outline two algorithms for 
computing $d(R,n,D)$ when $R$ is a maximal order of class number~$1$.
We have implemented these algorithms in Magma, and the resulting
programs are available at the URL mentioned in the 
introduction --- follow the links related to this paper, and download
the file \texttt{HermitianForms.magma}.  The entries in 
Tables~\ref{T:3}--\ref{T:11} reflect the output of these programs.

\begin{table}[t]
\begin{center}
\begin{tabular}{|r||r|r|r|r|c|r||r|r|r|r|c|r||r|r|r|r|}
\cline{1-5}\cline{7-11}\cline{13-17}
            &\multicolumn{4}{c|}{Rank $n$} && &\multicolumn{4}{c|}{Rank $n$}&& &\multicolumn{4}{c|}{Rank $n$}\\
\cline{2-5}\cline{8-11}\cline{14-17}
$D$& \phantom{---}\llap{$2$} & \phantom{---}\llap{$3$} & \phantom{---}\llap{$4$} & $5$ 
&\phantom{---}& 
$D$& \phantom{---}\llap{$2$} & \phantom{---}\llap{$3$} & \phantom{---}\llap{$4$} & $5$
&\phantom{---}& 
$D$& \phantom{---}\llap{$2$} & \phantom{---}\llap{$3$} & \phantom{---}\llap{$4$} & $5$ \\
\cline{1-5}\cline{7-11}\cline{13-17}
 1 &  1 &  1 &  1 &  1  && 11 &  3 &  2 &  2 & --- && 21 &  5 &  3 &  3 &  2  \\ 
 2 &  1 &  1 &  1 & --- && 12 &  4 &  3 &  2 &  2  && 22 &  5 &  2 &  3 & --- \\ 
 3 &  2 &  2 &  2 &  2  && 13 &  4 &  2 &  2 &  2  && 23 &  5 &  3 &  2 & --- \\
 4 &  2 &  2 &  2 &  2  && 14 &  3 &  2 &  2 & --- && 24 &  6 &  4 &  3 & --- \\ 
 5 &  2 &  1 &  2 & --- && 15 &  4 &  3 &  2 & --- && 25 &  5 &  3 &  3 &  2  \\ 
 6 &  3 &  2 &  2 & --- && 16 &  4 &  3 &  3 &  3  && 26 &  5 &  3 &  2 & --- \\ 
 7 &  2 &  2 &  2 &  2  && 17 &  4 &  3 &  2 & --- && 27 &  6 &  3 &  3 &  3  \\
 8 &  3 &  2 &  2 & --- && 18 &  5 &  3 &  3 & --- && 28 &  5 &  3 &  3 &  3  \\ 
 9 &  3 &  3 &  3 &  2  && 19 &  4 &  3 &  2 &  2  && 29 &  6 &  3 &  3 & --- \\
10 &  2 &  2 &  2 & --- && 20 &  4 &  3 &  2 & --- && 30 &  6 &  4 &  3 & --- \\
\cline{1-5}\cline{7-11}\cline{13-17}
\end{tabular}
\end{center}
\caption{Values of $d(R,n,D)$ for the quadratic order $R$ of 
         discriminant $-3$.  Dashes indicate values that we 
         have not computed.}
\label{T:3}
\end{table}

\begin{table}[t]
\begin{center}
\begin{tabular}{|r||r|r|r|r|c|r||r|r|r|r|c|r||r|r|r|r|}
\cline{1-5}\cline{7-11}\cline{13-17}
            &\multicolumn{4}{c|}{Rank $n$} && &\multicolumn{4}{c|}{Rank $n$}&& &\multicolumn{4}{c|}{Rank $n$}\\
\cline{2-5}\cline{8-11}\cline{14-17}
$D$& \phantom{---}\llap{$2$} & \phantom{---}\llap{$3$} & \phantom{---}\llap{$4$} & $5$ 
&\phantom{---}& 
$D$& \phantom{---}\llap{$2$} & \phantom{---}\llap{$3$} & \phantom{---}\llap{$4$} & $5$
&\phantom{---}& 
$D$& \phantom{---}\llap{$2$} & \phantom{---}\llap{$3$} & \phantom{---}\llap{$4$} & $5$ \\
\cline{1-5}\cline{7-11}\cline{13-17}
 1 &  1 &  1 &  2 &  1  && 11 &  4 &  2 &  2 & --- && 21 &  5 &  3 & ---& --- \\
 2 &  2 &  2 &  2 &  2  && 12 &  4 &  3 &  2 & --- && 22 &  5 &  3 & ---& --- \\
 3 &  2 &  1 &  2 & --- && 13 &  3 &  3 &  2 &  2  && 23 &  6 &  3 & ---& --- \\
 4 &  2 &  2 &  2 &  2  && 14 &  4 &  2 & ---& --- && 24 &  6 &  4 & ---& --- \\
 5 &  2 &  2 &  2 &  2  && 15 &  4 &  3 & ---& --- && 25 &  5 &  3 &  4 &  3  \\
 6 &  2 &  2 &  2 & --- && 16 &  4 &  4 &  4 &  3  && 26 &  6 &  4 &  3 &  3  \\
 7 &  3 &  2 &  2 & --- && 17 &  5 &  3 &  3 &  2  && 27 &  6 &  4 & ---& --- \\
 8 &  4 &  2 &  2 &  2  && 18 &  6 &  3 &  3 &  2  && 28 &  6 &  4 & ---& --- \\
 9 &  3 &  2 &  2 &  2  && 19 &  4 &  3 & ---& --- && 29 &  6 &  3 &  3 & --- \\
10 &  3 &  2 &  2 &  2  && 20 &  5 &  4 &  3 &  3  && 30 &  6 &  4 & ---& --- \\
\cline{1-5}\cline{7-11}\cline{13-17}
\end{tabular}
\end{center}
\caption{Values of $d(R,n,D)$ for the quadratic order $R$ of 
         discriminant $-4$.  Dashes indicate values that we 
         have not computed.}
\label{T:4}
\end{table}

\begin{table}[t]
\begin{center}
\begin{tabular}{|r||r|r|r|r|c|r||r|r|r|r|c|r||r|r|r|r|}
\cline{1-5}\cline{7-11}\cline{13-17}
            &\multicolumn{4}{c|}{Rank $n$} && &\multicolumn{4}{c|}{Rank $n$}&& &\multicolumn{4}{c|}{Rank $n$}\\
\cline{2-5}\cline{8-11}\cline{14-17}
$D$& \phantom{---}\llap{$2$} & \phantom{---}\llap{$3$} & \phantom{---}\llap{$4$} & $5$ 
&\phantom{---}& 
$D$& \phantom{---}\llap{$2$} & \phantom{---}\llap{$3$} & \phantom{---}\llap{$4$} & $5$
&\phantom{---}& 
$D$& \phantom{---}\llap{$2$} & \phantom{---}\llap{$3$} & \phantom{---}\llap{$4$} & $5$ \\
\cline{1-5}\cline{7-11}\cline{13-17}
 1 &  1 &  2 &  2 &  2  && 11 &  3 &  3 &  3 &  3  && 21 &  7 &  4 & ---& --- \\
 2 &  2 &  2 &  2 &  2  && 12 &  4 &  3 & ---& --- && 22 &  5 &  4 &  3 &  3  \\
 3 &  2 &  2 & ---& --- && 13 &  4 &  4 & ---& --- && 23 &  5 &  4 &  4 &  3  \\
 4 &  2 &  2 &  2 &  2  && 14 &  5 &  3 &  3 &  3  && 24 &  6 &  4 & ---& --- \\
 5 &  3 &  2 & ---& --- && 15 &  4 &  3 & ---& --- && 25 &  6 &  4 &  4 &  3  \\
 6 &  2 &  2 & ---& --- && 16 &  4 &  4 &  4 &  4  && 26 &  6 &  4 & ---& --- \\
 7 &  3 &  3 &  3 &  3  && 17 &  5 &  3 & ---& --- && 27 &  7 &  6 & ---& --- \\
 8 &  4 &  4 &  2 &  2  && 18 &  6 &  4 &  4 &  3  && 28 &  7 &  5 &  4 &  4  \\
 9 &  4 &  3 &  3 &  3  && 19 &  5 &  3 & ---& --- && 29 &  6 &  5 &  4 & --- \\
10 &  3 &  2 & ---& --- && 20 &  6 &  4 & ---& --- && 30 &  6 &  5 & ---& --- \\
\cline{1-5}\cline{7-11}\cline{13-17}
\end{tabular}
\end{center}
\caption{Values of $d(R,n,D)$ for the quadratic order $R$ of 
         discriminant $-7$.  Dashes indicate values that we 
         have not computed.}
\label{T:7}
\end{table}

\begin{table}[t]
\begin{center}
\begin{tabular}{|r||r|r|r|r|c|r||r|r|r|r|c|r||r|r|r|r|}
\cline{1-5}\cline{7-11}\cline{13-17}
            &\multicolumn{4}{c|}{Rank $n$} && &\multicolumn{4}{c|}{Rank $n$}&& &\multicolumn{4}{c|}{Rank $n$}\\
\cline{2-5}\cline{8-11}\cline{14-17}
$D$& \phantom{---}\llap{$2$} & \phantom{---}\llap{$3$} & \phantom{---}\llap{$4$} & $5$ 
&\phantom{---}& 
$D$& \phantom{---}\llap{$2$} & \phantom{---}\llap{$3$} & \phantom{---}\llap{$4$} & $5$
&\phantom{---}& 
$D$& \phantom{---}\llap{$2$} & \phantom{---}\llap{$3$} & \phantom{---}\llap{$4$} & $5$ \\
\cline{1-5}\cline{7-11}\cline{13-17}
 1 &  2 &  1 &  2 &  2  && 11 &  4 &  3 &  3 &  3  && 21 &  6 &  4 & ---& --- \\
 2 &  2 &  2 &  2 &  2  && 12 &  4 &  4 &  4 &  4  && 22 &  7 &  4 &  4 &  4  \\
 3 &  2 &  2 &  2 &  2  && 13 &  5 &  3 & ---& --- && 23 &  8 &  4 & ---& --- \\
 4 &  4 &  2 &  4 &  2  && 14 &  6 &  3 & ---& --- && 24 &  8 &  4 &  4 &  4  \\
 5 &  2 &  2 & ---& --- && 15 &  6 &  3 & ---& --- && 25 & 10 &  4 &  4 &  4  \\
 6 &  3 &  2 &  2 &  2  && 16 &  8 &  4 &  4 &  4  && 26 &  6 &  4 & ---& --- \\
 7 &  4 &  2 & ---& --- && 17 &  6 &  4 &  4 &  3  && 27 &  7 &  5 &  4 &  4  \\
 8 &  4 &  4 &  4 &  4  && 18 &  6 &  4 &  4 &  4  && 28 &  8 &  4 & ---& --- \\
 9 &  6 &  3 &  4 &  3  && 19 &  6 &  4 &  4 &  3  && 29 &  7 &  4 & ---& --- \\
10 &  4 &  3 & ---& --- && 20 &  6 &  4 & ---& --- && 30 &  8 &  5 & ---& --- \\
\cline{1-5}\cline{7-11}\cline{13-17}
\end{tabular}
\end{center}
\caption{Values of $d(R,n,D)$ for the quadratic order $R$ of 
         discriminant $-8$.  Dashes indicate values that we 
         have not computed.}
\label{T:8}
\end{table}

\begin{table}[t]
\begin{center}
\begin{tabular}{|r||r|r|r|r|c|r||r|r|r|r|c|r||r|r|r|r|}
\cline{1-5}\cline{7-11}\cline{13-17}
            &\multicolumn{4}{c|}{Rank $n$} && &\multicolumn{4}{c|}{Rank $n$}&& &\multicolumn{4}{c|}{Rank $n$}\\
\cline{2-5}\cline{8-11}\cline{14-17}
$D$& \phantom{---}\llap{$2$} & \phantom{---}\llap{$3$} & \phantom{---}\llap{$4$} & $5$ 
&\phantom{---}& 
$D$& \phantom{---}\llap{$2$} & \phantom{---}\llap{$3$} & \phantom{---}\llap{$4$} & $5$
&\phantom{---}& 
$D$& \phantom{---}\llap{$2$} & \phantom{---}\llap{$3$} & \phantom{---}\llap{$4$} & $5$ \\
\cline{1-5}\cline{7-11}\cline{13-17}
 1 &  2 &  1 &  2 &  3  && 11 &  6 &  4 &  4 &  5  && 21 &  9 &  5 & ---& --- \\
 2 &  1 &  2 & ---& --- && 12 &  7 &  4 &  4 &  4  && 22 & 11 &  5 & ---& --- \\
 3 &  2 &  2 &  3 &  3  && 13 &  5 &  4 & ---& --- && 23 &  6 &  5 &  4 &  4  \\
 4 &  4 &  3 &  3 &  3  && 14 &  5 &  4 & ---& --- && 24 &  8 &  5 & ---& --- \\
 5 &  4 &  2 &  4 &  3  && 15 &  6 &  4 &  4 &  4  && 25 & 10 &  5 &  5 &  5  \\
 6 &  3 &  3 & ---& --- && 16 &  8 &  4 &  5 &  4  && 26 &  7 &  5 & ---& --- \\
 7 &  3 &  3 & ---& --- && 17 &  6 &  5 & ---& --- && 27 &  8 &  6 &  5 &  5  \\
 8 &  4 &  4 & ---& --- && 18 &  7 &  5 & ---& --- && 28 &  9 &  5 & ---& --- \\
 9 &  6 &  4 &  4 &  4  && 19 &  8 &  5 & ---& --- && 29 & 10 &  5 & ---& --- \\
10 &  5 &  4 & ---& --- && 20 &  8 &  5 &  4 &  4  && 30 &  9 &  5 & ---& --- \\
\cline{1-5}\cline{7-11}\cline{13-17}
\end{tabular}
\end{center}
\caption{Values of $d(R,n,D)$ for the quadratic order $R$ of 
         discriminant $-11$.  Dashes indicate values that we 
         have not computed.}
\label{T:11}
\end{table}

Let $R$ be the maximal order of an imaginary quadratic field $K$ of
class number~$1$.  Our first algorithm will compute, for any rank $n$ 
and determinant~$D$, the value of $d(R,n,D)$.

Let $M$ be an $n$-by-$n$ Hermitian matrix with entries in~$R$ and let
$L$ be the corresponding Hermitian $R$-lattice. For each positive 
integer $i\le n$ we define the \emph{$i$'th successive $R$-minimum} of 
$L$ to be the smallest integer $N_i$ such that the elements of $L$ of 
length $N_i$ or less span a $K$-vector space of dimension at least~$i$.

Let the successive $R$-minima of $L$ be $N_1, \ldots, N_n$. Let $L_\ZZ$
be the $R$-lattice $L$ viewed as a rank-$2n$ lattice over~$\ZZ$, and 
let $M_1, \ldots, M_{2n}$ be the successive minima of~$L_\ZZ$.  Then
\begin{align*}
  N_1  &=    M_1\\
  N_2  &\le  M_3\\
  N_3  &\le  M_5
\end{align*}  
and so on, so that 
\begin{align*}
  (N_1 \cdots N_n)^2  & =  N_1^2  N_2^2\cdots N_n^2\\
                      &\le (M_1 M_2)  (M_3 M_4)\cdots (M_{2n-1} M_{2n}).
\end{align*}     
Arguing as in the proof of Lemma~\ref{L:quadform}, we find that
\[\det L_\ZZ = (\det M)^2 (\left|\disc R\right|/4)^n,\]
and combining this with~\cite{Cassels}*{Thm~12.2.2, p.~262} we find
that
\[ N_1 \cdots N_n \le \gamma_{2n}^n \, (\det M) \left(\left|\disc R\right|/4\right)^{n/2}.\]
From this, we obtain an upper bound on $N_1$.

Let this initial upper bound be $B$.  We let a variable $s$ take on 
successive values $B$, $B-1$, and so on, down to $1$.  For a given 
value of~$s$, we try to construct $R$-lattices whose successive 
$R$-minima are all greater than or equal to $s$.  The first $s$ for
which we succeed in constructing such a lattice will be the value of
$d(R,n,D).$  We attempt to construct an $R$-lattice with $N_1=s$ as 
follows:

Suppose the successive $R$-minima of $L$ are all $s$ or larger.  The
product bound above gives us a finite set of values of the $N_i$ to 
consider.  For each possible set of~$N_i$, suppose we have an 
$R$-lattice $L$ with those minima.  Consider the sublattice $L'$ of $L$
generated by vectors giving those minima.  The Gram matrix for $L'$
will have the $N_i$ on its diagonal, and we get bounds for the other 
entries from the fact that each rank-$2$ sublattice of $L'$ is positive 
definite and has no vectors of length less than $s$.  So we can 
enumerate all of the $L'$, and then see whether any of the $L'$ have
superlattices with discriminant $D$ and with no vectors of length less
than~$s$.

Even without a formal complexity analysis, it is not hard to see that
the work required to run the algorithm outlined above grows at least 
on the order of 
\[ (\gamma_{2n}/2) ^ {(n^2-n)/2}\,  D^{(n-1)/2} \, \left|\disc R \right|^{(n^2-n)/4}.\]
We have implemented the algorithm for $n = 2, 3, 4$, and $5$
in Magma, in the routines \texttt{FindMinimum2}, \texttt{FindMinimum3}, 
\texttt{FindMinimum4}, and \texttt{FindMinimum5}, respectively.  
In practice, for $n=5$ our implementation took more time to
run than we were willing to wait, and for $n=4$ we only ran the 
algorithm for the orders of discriminants $-3$ and~$-4$.  All of the
values in the tables for~$n=5$, and most of the values for~$n=4$, 
came from running our second algorithm, which computes $d(R,n,D)$ only
in the case where $D$ is the norm of an element of $R$.

We must introduce some additional notation before outlining the second 
algorithm.  Throughout, $R$ will continue to denote the maximal order 
of an imaginary quadratic field $K$ with class number~$1$.

Let $L$ be the lattice $R^n$, viewed as a subset of $K^n$.  For every 
prime ideal $\gothp$ of $R$, we fix a finite set $\calS_\gothp$ of
matrices in $M_n(R)$ such that
\[\{P^{-1} L : P \in \calS_\gothp\}\]
is the complete set of the superlattices $M\supset L$ in $K^n$ such
that $M/L\cong R/\gothp$ as $R$-modules.  For example, if $\pi$ is a 
generator of the principal ideal $\gothp$, and if $X\subset R$ is a 
set of representatives for the residue classes of $\gothp$, then one
choice for $\calS_\gothp$ would be
\[
\left\{
\left[\begin{matrix} \pi    & 0      & 0      & \cdots & 0      \\
                     x_2    & 1      & 0      & \cdots & 0      \\
                     x_3    & 0      & 1      & \cdots & 0      \\
                     \vdots & \vdots & \vdots &        & \vdots \\
                     x_n    & 0      & 0      & \cdots & 1      \\
\end{matrix}\right] \col x_i \in X\right\}
\cup
\left\{
\left[\begin{matrix} 1      & 0      & 0      & \cdots & 0      \\
                     0      & \pi    & 0      & \cdots & 0      \\
                     0      & x_3    & 1      & \cdots & 0      \\
                     \vdots & \vdots & \vdots &        & \vdots \\
                     0      & x_n    & 0      & \cdots & 1      \\
\end{matrix}\right] \col x_i \in X\right\}
\cup\cdots
\]
We say that two Hermitian matrices $A$ and $B$ in $M_n(R)$ are
\emph{isomorphic} to one another if there is an invertible
$C\in M_n(R)$ such that $A = C^* B C$, where $C^*$ is the conjugate
transpose of $C$.  We also fix a finite set $\calU$ of representatives
of the isomorphism classes of unimodular Hermitian matrices in $M_n(R)$.
For the $R$ and $n$ we will be considering, Schiemann~\cite{Schiemann} 
has computed such sets $\calU$.

\begin{lemma}
\label{L:prime-mats}
If $A$ is a matrix in $M_n(R)$ whose determinant generates a prime 
ideal $\gothp$, then there is an element $P$ of $\calS_{\gothp}$ and 
an invertible $C\in M_n(R)$ such that $A = CP$. 
\end{lemma}

\begin{proof}
The lattice $M=A^{-1}L$ is a superlattice of $L$ such that
$M/L \cong R/\gothp$, so there is a $P\in \calS_\gothp$ so that
$A^{-1}L = P^{-1}L$.  If we set $C = AP^{-1}$ then $CL=L$, so $C$ is an
invertible element of $M_n(R)$.
\end{proof}

\begin{lemma}
\label{L:decomposition}
Suppose $A$ is a Hermitian matrix in $M_n(R)$ whose determinant is 
equal to $x\xb$ for some $x$ in $R$.  Write the ideal $xR$ as a
product $\gothp_1\cdots\gothp_r$ of prime ideals.  Then $A$ is 
isomorphic to a product
\[P_1 P_2 \cdots P_r U P_r^* \cdots P_2^* P_1^*\]
where $U\in \calU$ and $P_i\in\calS_{\gothp_i}$.
\end{lemma}

\begin{proof}
We prove this statement by induction on the number $r$ of prime factors
of~$xR$.  The statement is certainly true when $r=0$, because in that
case $\det A = 1$ and $A$ is isomorphic to one of the matrices
in~$\calU$.

Suppose $r>0$, and let $\gothq = \gothpbar_1$.  Let $M = A^{-1}L$
and let $G$ be the finite $R$-module $M/L$, whose cardinality is 
$(\det A)^2$.  Pick $\delta\in R$ with $\delta^2 = \disc R$, so that 
$\delta$ is a purely imaginary generator of the different of~$R$.  We 
define a pairing $b\col G\times G\to\QQ/\ZZ$ by setting 
\[b(x,y) = \Tr(x^*Ay/\delta) \bmod \ZZ; \]
it is easy to check that this pairing is well-defined, and by using the
fact that $(1/\delta)L$ is the trace dual of $L$, we see that $b$ is
nondegenerate.  Note also that $b$ is alternating, and that for all 
$r\in R$ and $x,y\in G$ we have $b(rx,y) = b(x,\rbar y)$, so that in
the terminology of~\cite{Howe1995}, $b$ is \emph{semi-balanced}.

The ideal $\gothq$ occurs in the Jordan--H\"older decomposition for the
$R$-module $G$, so the $\gothq$-torsion $T$ of $G$ is nontrivial.  We 
claim that we can find a $1$-dimensional $R/\gothq$-vector subspace of
$T$ on which the pairing $b$ is identically~$0$.  If 
$\gothq\neq\gothqbar$ this follows 
from~\cite{Howe1995}*{Lem.~7.2, p.~2378}.  If $\gothq = \gothqbar$ and
the dimension of $T$ as a $R/\gothq$-vector space is at least~$2$, then
this follows from~\cite{Howe1995}*{Lem.~7.3, p.~2378}. If 
$\gothq = \gothqbar$ and $T$ is $1$-dimensional, we let $U$ be the
$\gothq$-power torsion of the $R$-module $G$. Since $\gothq^2$ divides
$\det A$, the $R$-module $U$ is strictly larger than $T$, and the 
annihilator of $T$ in $U$ is nontrivial.  Thus, the $\gothq$-torsion of
the annihilator of $T$ must be $T$ itself, so $b$ restricted to $T$ is
trivial.  This proves the claim.

Let $N$ be the sublattice of $M$ consisting of elements that reduce 
modulo $L$ to elements of $T$.  Then $N$ is a superlattice of $L$ with
the property that $N/L \cong R/\gothq$, so there is an element $Q$ of
$\calS_\gothq$ such that $N = Q^{-1}L$.

Let $x$ and $y$ be arbitrary elements of~$L$. Since $b$ is trivial on
$N/L$, we have
\[ \Tr((Q^{-1}x)^* A (Q^{-1}y) /\delta) \in \ZZ,\]
so that
\[ \Tr(x^* Q^{-1*} A Q^{-1} y  /\delta) \in \ZZ
                \quad\textup{for all $x,y\in L$.}\]
If we set $B = Q^{-1*} A Q^{-1}$, we see that $B y /\delta$ must lie in
the trace dual of $L$, which is $(1/\delta)L$, so $B$ must send $L$ 
to~$L$.  In other words, the entries of $B$ must all be elements
of~$R$. This shows that $A = Q^* B Q$ for a Hermitian matrix $B$ in 
$M_n(R)$ whose determinant can be written $y\yb$ for an element $y$
of $R$ with $yR = \gothp_2\cdots\gothp_r$.

Applying the induction hypothesis, we find that
\[ B = C^* P_2 P_3 \cdots P_r U P_r^* \cdots P_3^* P_2^* C \]
where each $P_i$ lies in $\calS_{\gothp_i}$, where $U$ lies in $\calU$,
and where $C$ is an invertible element of $M_n(R)$.  Thus
\[ A = Q^* C^* P_2 P_3 \cdots P_r U P_r^* \cdots P_3^* P_2^* C Q.\]
Note that $Q^*C^*$ is a matrix whose determinant generates the prime 
ideal $\gothp_1$, so by Lemma~\ref{L:prime-mats} there is an element 
$P_1$ of $\calS_{\gothp_1}$ and an invertible $D\in M_n(R)$ such that 
$Q^*C^*= D P_1$. Thus, we find that
\[ A = D P_1 P_2 P_3 \cdots P_r U P_r^* \cdots P_3^* P_2^* P_1^* D^*.\]
In other words, $A$ is isomorphic to a product as in the statement of 
the lemma.
\end{proof}

Lemma~\ref{L:decomposition} thus gives us an easy way to enumerate all
of the isomorphism classes of Hermitian matrices in $M_n(R)$ of a given
determinant, provided that the determinant is a norm. For each 
isomorphism class, we can compute the shortest vector using standard 
techniques that are built into Magma, and in this way we can compute an
upper bound on the lengths of the short vectors of such matrices.  This
completes the proof of Lemma~\ref{L:sharp}.
\end{proof}

\begin{remark}
The proof of Lemma~\ref{L:decomposition} is very similar to the proof
of~\cite{Howe1995}*{Prop.~7.1, p.~2378}, and indeed we initially
thought of Lemma~\ref{L:decomposition} not as a statement about the 
decomposition of Hermitian matrices but rather as a statement about
the decomposition of non-principal polarizations of abelian varieties.
\end{remark}

As an application of the sharp values of $d(R,n,D)$, we prove a 
generalization of~\cite{LauterSerre2002}*{Thm.~4, pp.~95--96}.

\begin{proposition}
\label{P:LauterSerre}
Let $q$ be a prime power, let $m=\lfloor 2\sqrt{q}\rfloor$, and suppose
there is an elliptic curve $E$ over $\Fq$ with trace $-m$.  Let $F$ be
an arbitrary elliptic curve over~$\Fq$.  Then there is no Jacobian over
$\Fq$ isogenous to $E^{g-1}\times F$ if $m^2-4q$, $g$, and the trace of
$F$ lie in the following table\textup{:}
\begin{center}
\renewcommand{\arraystretch}{1.2}
\begin{tabular}{|rrr|}
\hline
$m^2-4q$ &\quad $g$ & \quad \textup{trace} $F$\\ \hline
  $-3$   & $3$ & $-m+2$\\
         & $4$ & $-m+2$\\
         & $4$ & $-m+5$\\
         & $5$ & $-m+2$\\ \hline
  $-4$   & $3$ & $-m+3$\\ \hline
 $-11$   & $3$ & $-m+2$\\ \hline
\end{tabular}
\end{center}
\end{proposition}

\begin{proof}
Suppose $m^2-4q$, $g$, and the trace of $F$ lie in the table, and
suppose $C$ is a curve of genus $g$ over $\Fq$ whose Jacobian is 
isogenous to $E^{g-1}\times F$.  Write the trace of $F$ as $-m + f$.
Then the reduced resultant of the real Weil polynomials of $E$ and $F$
is $f$, and the largest group scheme that can be embedded in the 
$f$-torsion of an elliptic curve isogenous to $F$ has order $f^2$.  
Pulling the principal polarization of $\Jac C$ back to $E^{g-1}$ gives
us a polarization of degree $f^2$, corresponding to a positive definite
Hermitian matrix, with determinant $f$, over the quadratic order $R$ of
discriminant $m^2 - 4q$.  Looking in Tables~\ref{T:3},~\ref{T:4}, 
and~\ref{T:11}, we find that for all of the cases listed in the 
proposition, the value of $d(R,g-1,f)$ is~$1$, so that the Hermitian
form associated to this matrix has a vector of length~$1$.  This
vector gives us an embedding of $E$ into $E^{g-1}$; let $\psi$ be the
composition of this embedding with the map $E^{g-1}\to \Jac C$.  Then
Lemma~\ref{L:Enpols} shows that $\psi$ pulls the principal polarization
of $C$ back to the principal polarization of $E$, and 
Lemma~\ref{L:CtoE} shows that there is a map of degree~$1$ from $C$
to~$E$.  This is clearly impossible, so there must not be a $C$ with 
Jacobian isogenous to $E^{g-1}\times F$.
\end{proof}

\section{Galois descent}
\label{S:Galois}

For some isogeny classes $\calC$ of abelian varieties over a finite
field~$k$, one can show that every principally-polarized variety in 
$\calC$ can be defined over a subfield $k_0$ of~$k$; it follows that
every Jacobian in $\calC$ comes from a curve that can be defined 
over~$k_0$.  This reduces the problem of determining whether there are
Jacobians in $\calC$ to the problem of determining whether there are
Jacobians in a collection of isogeny classes over a smaller field.
This idea was used in~\cite{Serre:notes}*{pp.~Se42--Se43} and
in~\cite{LauterSerre2001}; the appendix to the latter paper describes
some methods for determining whether principally-polarized varieties
can be defined over subfields.

In this section we give a simple necessary and sufficient condition for
determining whether the entire category of varieties in an 
\emph{ordinary} isogeny class $\calC$ can be descended in this way.  To
begin, we set some notation and make a formal definition.

Let $k_0$ be a finite field, $\calC_0$ an isogeny class of abelian 
varieties over $k_0$, and $k$ a finite extension of $k_0$, say of 
degree $e$ over~$k_0$.  Base extension by $k/k_0$ takes the isogeny 
class $\calC_0$ to an isogeny class $\calC$ over $k$, and the base
extension functor respects properties such as the degrees of isogenies,
the duality of varieties, and whether or not an isogeny is a 
polarization.

\begin{definition}
We say that $\calC$ \emph{descends to $\calC_0$} if base extension
induces an equivalence between the category of abelian varieties in 
$\calC_0$ and the category of abelian varieties in $\calC$.
\end{definition}

Let $\calC$ be an arbitrary isogeny class of ordinary abelian varieties
over a finite field~$k$, and let $A$ be any variety in $\calC$.  Let 
$\pi$ and $\pibar$ be the Frobenius and Verschiebung endomorphisms 
of~$A$, respectively; then the subring $R := \ZZ[\pi,\pibar]$ of 
$\End A$ is contained in the center $K$ of the ring 
$E :=(\End A)\otimes\QQ$.  Up to isomorphism, the ring $R$ and the 
$\QQ$-algebras $K$ and $E$ do not depend on the choice of $A$; we 
denote them by $R_\calC$, $K_\calC$, and $E_\calC$.  The algebra $K$ is
a product of CM fields, and $R$ is an order in~$K$.  Furthermore,
complex conjugation on $K$ sends $\pi$ to~$\pibar$.

\begin{theorem}
\label{T:Galois}
Let $\calC$ be an isogeny class of ordinary abelian varieties over a 
finite field $k$ that contains an index-$e$ subfield $k_0$.

If there is an element $\pi_0$ of $R_\calC$ such that $\pi = \pi_0^e$,
then the characteristic polynomial of $\pi_0$ \textup{(}as an element
of $E_\calC$\textup{)} is the Weil polynomial for an isogeny class
$\calC_0$ of abelian varieties over~$k_0$, and $\calC$ descends 
to~$\calC_0$.

Conversely, if $\calC$ descends to an isogeny class $\calC_0$ 
over~$k_0$, then there is a $\pi_0$ in $R_\calC$ whose characteristic
polynomial \textup{(}as an element of $E_\calC$\textup{)} is equal to 
the Weil polynomial for $\calC_0$ and such that $\pi = \pi_0^e$.
\end{theorem}

\begin{proof}
Let $q = \#k$ and $q_0 = \#k_0$, so that $q = q_0^e$, and let $p$ be 
the characteristic of~$k$.  Let $R = R_\calC$ and $K = K_\calC$.

Suppose there is an element $\pi_0$ of $R$ such that $\pi = \pi_0^e$,
and let $g$ be its characteristic polynomial.  The product 
$\pi_0\pibar_0$ of $\pi_0$ with its complex conjugate is totally
positive and real, and since
\[ (\pi_0\pibar_0)^e = \pi\pibar = q = q_0^e \]
we see that $\pi_0\pibar_0 = q_0$.  This shows that all of the complex
roots of $g$ have magnitude~$\sqrt{q_0}$, so all of the roots of $g$ 
are $q_0$-Weil numbers. To show that the corresponding isogeny class of
varieties is ordinary, we must show that for every homomorphism 
$\varphi$ of $K_\calC$ to $\QQ_p$, one of the numbers $\varphi(\pi_0)$
and $\varphi(\pibar_0)$ is a unit and the other is not.  But this
follows from the fact that for each $\varphi$, one of the numbers 
$\varphi(\pi_0^e)$ and $\varphi(\pibar_0^e)$ is a unit and the other is
not, which is true because $\calC$ is ordinary.  The Honda--Tate theorem
then shows that $g$ is the Weil polynomial of an isogeny class 
$\calC_0$ of ordinary abelian varieties over~$k_0$.

Let $f$ be the Weil polynomial of $\calC$, and let 
$f = f_1^{e_1}\cdots f_r^{e_r}$ be its factorization into powers of
distinct irreducibles.  Each $f_i$ defines a CM-field~$K_i$, and $K$ is
the product of these~$K_i$. Likewise, we can write the factorization of 
$g$ as $g = g_1^{e_1}\cdots g_r^{e_r}$, where each $g_i$ also 
defines~$K_i$.

Let $R_0$ be the ring $\ZZ[\pi_0,\pibar_0]$, so that 
$R_0 = R_{\calC_0}$.  Deligne's theorem on ordinary abelian 
varieties~\cite{Deligne1969} shows that the category of abelian 
varieties in $\calC_0$ is equivalent to the category of finitely 
generated $R_0$-modules that can be embedded in 
$V:=K_1^{e_1}\times \cdots \times K_r^{e_r}$ as submodules whose images
span $V$ as a $\QQ$-vector space.  (The first author~\cite{Howe1995}
has shown how dual varieties and polarizations can be interpreted in 
this category of $R_0$-modules.)  Likewise, the category of abelian
varieties in $\calC$ is equivalent to the category of finitely 
generated $R$-modules that can be embedded in $V$ as submodules whose
images span $V$ as a $\QQ$-vector space. The base extension functor 
sends an $R_0$-module $M$ to the same module, viewed as a module over 
the subring $\ZZ[\pi_0^e,\pibar_0^e] = R$ of~$R_0$.  But since $\pi_0$
and $\pibar_0$ lie in $R$, we have $R=R_0$, so base extension gives an
equivalence of categories.  This shows that $\calC$ descends 
to~$\calC_0$, and proves the first statement of the theorem.

Now assume that $\calC$ is an ordinary isogeny class that descends to
an isogeny class $\calC_0$ over~$k_0$.  Clearly $\calC_0$ must also be
ordinary.  Let $R_0 = R_{\calC_0}$, and let $\pi_0\in R_0$ be the
Frobenius for~$\calC_0$.  Then the Frobenius $\pi$ for the isogeny
class $\calC$ is~$\pi_0^e$, and the ring $R = R_\calC$ is isomorphic to
the subring $\ZZ[\pi_0^e,\pibar_0^e]$ of~$R_0$.  

We know that $R_0$ is contained in the center of the endomorphism ring
of every variety in $\calC_0$, and it follows 
from~\cite{Waterhouse1969}*{Thm.~7.4, p.~554} or 
from~\cite{Deligne1969} that there exist varieties in $\calC_0$ whose 
endomorphism rings have centers \emph{equal} to~$R_0$.  Thus, $R_0$ can
be characterized as the smallest ring that occurs as the center of the
endomorphism ring of a variety in $\calC_0$.  Likewise, $R$ is the 
smallest ring that occurs as the center of the endomorphism ring of a
variety in~$\calC$. Since we are assuming that base extension gives an
equivalence of categories from $\calC_0$ to $\calC$, we find that we
must have $R_0\cong R$.  It follows that the natural inclusion
$R = \ZZ[\pi_0^e,\pibar_0^e]\subset R_0$ is an isomorphism, so $R$ 
contains an element $\pi_0$ whose characteristic polynomial is the Weil
polynomial for $\calC_0$ and with $\pi = \pi_0^e$.
\end{proof}

\begin{remark}
We could also have proven Theorem~\ref{T:Galois} by using
Th\'eor\`emes~6 and~7 from~\cite{LauterSerre2001}*{\S\S4,5} to show
that each variety in $\calC$, and each polarization of each variety 
in~$\calC$, descends to~$k_0$.  However, we felt that the argument
above, which gives us an entire equivalence of categories between
the isogeny classes $\calC$ and $\calC_0$ all at once, was worth the
small additional effort of introducing Deligne modules into the proof.
\end{remark}

If $A$ is an abelian variety over a finite field~$\Fq$, the 
\emph{standard quadratic twist} $A'$ of $A$ is the twist of $A$ 
corresponding to the element of the cohomology set 
$H^1(\Gal \Fqbar/\Fq, \Aut A)$ represented by the cocycle that sends
the $q$th-power Frobenius automorphism of $\Fqbar$ to the 
automorphism $-1$ of $A$.  Suppose $\calC$ is an isogeny class of 
abelian varieties over a finite field. The \emph{quadratic twist} 
$\calC'$ of $\calC$ is the isogeny class consisting of the standard 
quadratic twists of the elements of $\calC$.  If the Weil polynomial 
of $\calC$ is~$f(x)$, then the Weil polynomial of $\calC'$ is $f(-x)$.

The next result shows how Theorem~\ref{T:Galois} can help us show there
are no Jacobians in an isogeny class. 

\begin{theorem}
\label{T:descent}
Suppose $\calC$ is an isogeny class of ordinary abelian varieties over 
a finite field $k$ that descends to an isogeny class $\calC_0$ over a
subfield $k_0$ of $k$ of index~$e$, and suppose $C$ is a curve over $k$
whose Jacobian lies in $\calC$.
\begin{enumerate}
\item If $e$ is odd, then $C$ has a model over $k_0$ whose Jacobian 
      lies in $\calC_0$.
\item If $e$ is even, then $C$ has a model over $k_0$ whose Jacobian
      lies either in $\calC_0$ or in the quadratic twist of $\calC_0$.
\end{enumerate}
\end{theorem}

Thus, to show that there are no Jacobians in $\calC$, it suffices to 
show there are no Jacobians in $\calC_0$ and, if $e$ is even, in the
quadratic twist of $\calC_0$.

\begin{proof}
If $C$ is hyperelliptic, Th\'eor\`eme~4 of the appendix 
to~\cite{LauterSerre2001} shows that $C$ has a model over $k_0$ whose
Jacobian lies in $\calC_0$. If $C$ is not hyperelliptic, Th\'eor\`eme~5
of the same appendix shows that $C$ has a model $C_0$ over $k_0$ whose
Jacobian has an $\eps$-twist that lies in $\calC_0$, where $\eps$ is a
homomorphism from $\Gal k/k_0$ to $\{\pm1\}$.  If $e$ is odd $\eps$ 
must be trivial, so $\Jac C_0$ lies in $\calC_0$, and we get 
statement~(1) of the theorem. If $e$ is even, then the $\eps$-twist is
either trivial or the standard quadratic twist, and we get 
statement~(2).
\end{proof}

\section{Magma implementation}
\label{S:Magma}

As we indicated in the introduction, we have implemented our various
tests in Magma. The main program is 
\texttt{isogeny{\us}classes(q,g,N)}, which takes as input a prime 
power~$q$, a genus~$g$, and a desired number of points~$N$.  Using the
algorithm outlined in~\cite{HoweLauter2003}, we enumerate all of the
monic degree-$g$ polynomials in $\ZZ[x]$ whose leading terms are 
$x^g + (N-q-1) x^{g-1}$ and all of whose roots are real numbers of 
absolute value at most $2\sqrt{q}$.  This set of polynomials includes 
the set of real Weil polynomials of Jacobians of curves with $N$ 
points.  For each such polynomial~$f$, the program runs the subroutine
\texttt{process{\us}isogeny{\us}class}, which answers `no', `maybe', or
`yes' to the question ``Is there a Jacobian whose real Weil polynomial
is equal to~$f$?''

The procedure \texttt{process{\us}isogeny{\us}class}, when supplied 
with a polynomial, performs the following steps:

\begin{enumerate}
\item  The procedure checks whether the polynomial corresponds to an
       isogeny class of abelian varieties; that is, it checks whether
       the polynomial satisfies the conditions of the Honda--Tate 
       theorem~\cite{Tate1968}*{Th\'eor\`eme 1, p.~96}. If not, the 
       answer to the question is `no'.
\item  If the dimension of the isogeny class is $2$, the procedure
       checks whether it meets the conditions of the
       Howe/Nart/Ritzenthaler classification of $2$-dimensional 
       isogeny classes that contain Jacobians~\cite{HNR}. The answer 
       to the question is `yes' or `no', accordingly.
\item  The procedure checks whether the Weil polynomial predicts a 
       non-negative number of degree-$d$ places for all $d$ less than
       or equal to the genus. (The isogeny classes returned by 
       \texttt{isogeny{\us}classes()} have this property, but for 
       isogeny classes that arise recursively in some of the following
       steps, this condition must be checked.) If not, the answer to 
       the question is `no'.
\item  If the isogeny class is maximal (that is, if $N$ is equal to the
       Weil bound for genus-$g$ curves over $\Fq$), the procedure 
       checks whether the results of Korchmaros and
       Torres~\cite{KorchmarosTorres} forbid the existence of a curve
       with Jacobian in the isogeny class.  If so, the answer to the 
       question is `no'.
\item  The procedure checks whether the isogeny class factors as an 
       ordinary isogeny class times the class of a power of a
       supersingular elliptic curve with all endomorphisms defined.
       If so, it checks to see whether Theorem~\ref{T:SSfactor} shows 
       that the isogeny class does not contain a Jacobian. If so, the 
       answer to the question is `no'.
\item  The procedure uses Theorem~\ref{T:Galois} to check whether the 
       isogeny class can be descended to an isogeny class over a 
       subfield.  If so, the procedure uses Theorem~\ref{T:descent} to 
       recurse, and checks  whether the associated isogeny classes over
       the subfield contain Jacobians.  If they do not, then the answer
       to the question is `no'.
\item  The procedure checks whether the real Weil polynomial can be
       split into two factors whose resultant is $1$.  A result of 
       Serre (see~\cite{HoweLauter2003}*{Thm.~1(a), p.~1678}) says that
       no Jacobian can lie in such a class, so if there is such a 
       splitting, the answer to the question is `no'.
\item  Using Proposition~\ref{P:RR}, the procedure checks whether the
       real Weil polynomial can be split into two factors whose gluing
       exponent is $2$.  In this case, any curve whose Jacobian lies in
       the isogeny class must have an involution (Theorem~\ref{T:GE}),
       and so must be a double cover of a curve $D$ whose real Weil
       polynomial $g$ can be determined up to at most two
       possibilities.  If a contradiction can be deduced from this,
       either using Lemma~\ref{L:ramification} (below) or by showing
       recursively that there is no curve with real Weil polynomial 
       equal to~$g$, the answer to the question is `no'. 
\item  The procedure checks to see whether Proposition~\ref{P:newec},
       or a refinement using our tables of maximal lengths of short 
       vectors of Hermitian forms, can be used to deduce the existence
       of a map of known degree $n$ from any curve $C$ with real Weil
       polynomial $f$ to an elliptic curve $E$ with a known trace.  If
       such a map can be shown to exist, and if its existence leads to
       a contradiction (either by using Lemma~\ref{L:ramification} if
       $n=2$, or by noting that $\#C(\Fq) > n \#E(\Fq)$), the answer to
       the question is `no'.
\item  If at this point the question has not yet been answered, the 
       answer defaults to `maybe', because we have no proof that the
       answer is `no', and we do not know that the answer is `yes'.
\end{enumerate}

To decide whether there is a problem with there being a double cover
from a curve $C$ whose Jacobian lies in an isogeny class $\calC_1$ to a
curve $D$ whose Jacobian lies in an isogeny class $\calC_2$, we use the
following lemma:

\begin{lemma}
\label{L:ramification}
Suppose $C$ and $D$ are curves over $\Fq$ of genus $g_C$ and $g_D$, 
respectively, and for each $i$ let $a_i$ and $b_i$ denote the number of
places of degree $i$ on $C$ and on $D$, respectively.  Suppose 
$\varphi\col C\to D$ is a map of degree~$2$.  Let $r$ denote the number
of geometric points of $D$ that ramify in the double cover, and let 
$r_1$ denote the number of $\Fq$-rational points of $D$ that ramify in
the double cover.
\begin{enumerate}
\item We have $2b_1 - 2a_2 - a_1 \le r_1 \le 2b_1-a_1$.
\item We have $r_1 \equiv a_1\bmod 2$ and $r_1\ge 0$.
\item We have $r \ge r_1 + \sum_{1< d\le g_C, d \textup{\ odd}, a_d \textup{\ odd}} d.$
\item If $q$ is even, then $r = \tworank C - 2(\tworank D) + 1$ 
      and $r\le g_C - 2g_D + 1$.        
\item If $q$ is odd, then $r = 2 g_C - 4 g_D + 2$.
\end{enumerate}
\end{lemma}

\begin{proof}
Let $s_1$ and $i_1$ be the number of rational points of $D$ that split
and are inert (respectively) in $\varphi$.  Then we have
\[ 
   b_1 = s_1 + i_1 + r_1, \quad 
   a_1 = 2 s_1 + r_1, \quad\text{and}\quad
   a_2 \ge i_1. 
\]
These relations lead to statements~(1) and~(2).

If $d$ is odd, then the number of degree-$d$ places on $C$ is equal to
twice the number of splitting degree-$d$ places on~$D$, plus the number
of ramifying degree-$d$ places on~$D$.  If $d$ and $a_d$ are both odd,
then there must be at least one degree-$d$ place of $D$ that ramifies.
This leads to statement~(3).

If $q$ is even, then the Deuring--Shafarevich formula and the
Riemann--Hurwitz formula show that
\[
  r = \tworank C - 2(\tworank D) + 1 \quad\text{and}\quad
  r\le g_C - 2g_D + 1,
\]
respectively. This is statement~(4). If $q$ is odd, then the
Riemann--Hurwitz formula shows that $r = 2 g_C - 4 g_D + 2$, 
which is~(5).
\end{proof}

Lemma~\ref{L:ramification} gives a simple way of testing whether there
is a problem with the existence of a degree-$2$ map $C\to D$, where 
$\Jac C$ and $\Jac D$ lie in known isogeny classes.  If we find 
contradictory statements about $r$ and $r_1$, we know there is a 
problem with there being such a double cover. We note that in the Magma
routines from our earlier paper, we did not make use of all of the 
inequalities listed in Lemma~\ref{L:ramification}, so this is another 
place where the new program improves upon the old.

Finally, in some circumstances we can deduce that a curve with $N$
points must be a degree-$n$ cover of an elliptic curve $E$, for some 
$n>2$. This will lead to a contradiction if $N > n \#E(\Fq)$.

Using these criteria, the procedure 
\texttt{process{\us}isogeny{\us}class} decides whether it can deduce a
contradiction from the existence of a Jacobian in a given isogeny
class.

\section{New results and applications}
\label{S:results}

Our new program improved the best known upper bound on $N_q(g)$ for
more than $16\%$ of the $(q,g)$-pairs in the 2009 version of the tables
of van der Geer and van der Vlugt; this improvement is in addition to
the improvements that came from our earlier paper.  In this section we
present a sample of some of these new results to indicate how the 
theorems from earlier in the paper come into play. We also give some
examples that show that the information obtained from our program,
combined with further analysis, can be used to restrict the possible
Weil polynomials of curves with a given number of points.

\subsection{Proof that $N_9(12)\le 61$.}
Consider the case $q=9$ and $g=12$.  The Ihara bound says that 
$N_9(12)\le 63$, and the Magma program we wrote for our earlier 
paper~\cite{HoweLauter2003} showed that in fact $N_9(12)\le 62$. Our
new program shows that $N_9(12)\le 61$; in this subsection we explain
how our new techniques eliminate  cases that our old techniques could
not.

Our old program could show that if a genus-$12$ curve over $\FF_9$ has
$62$ points, then its real Weil polynomial is either
\[ 
   (x + 2) (x + 4)^6 (x + 5)^4 (x + 6)
   \text{\quad or\quad }
   (x + 4)^8 (x + 6)^2 (x^2 + 8x + 14).
\]   
Here is how our new program shows that the first of these polynomials
is not the real Weil polynomial of a curve.

Let $E$ be the unique elliptic curve over $\FF_9$ with $15$ points, so
that $E$ has trace $-5$.  Using the gluing exponent, or even just using
the resultant, one can show that if $A$ is a principally polarized
abelian variety with real Weil polynomial equal to
$(x + 2) (x + 4)^6 (x + 5)^4 (x + 6)$, then the principal polarization
of $A$ pulls back to a polarization of $E^4$ of degree $1$ or 
degree~$3^2$. Using Table~\ref{T:11} we find that this polarization of
$E^4$ pulls back to give a polarization of degree at most $3^2$ on~$E$,
and Lemma~\ref{L:CtoE} shows that if $A$ is the Jacobian of a 
curve~$C$, then $C$ has a map of degree at most $3$ to~$E$.  But then
$C$ could only have at most $3$ times the number of rational points
that $E$ has, so that $\#C(\FF_9)\le 45$, and in particular $C$ does
not have~$62$ points. This eliminates the first of the two polynomials
above.

We turn now to the second polynomial. The smallest resultant that we
can obtain between two complementary factors of the radical of the
second real Weil polynomial is $4$, and in fact we get this resultant
from each of the three possible splittings. However, the reduced 
resultant of $(x+4)(x+6)$ and $(x^2 + 8x + 14)$ is~$2$, so using 
Theorem~\ref{T:GE} we find that a genus-$12$ curve $C$ with the given
real Weil polynomial must be a double cover of a genus-$2$ curve $D$
with real Weil polynomial equal to $x^2 + 8x + 14$. From its real Weil
polynomial we see that the curve $D$ has only $18$ rational points, so
a double cover of $D$ can have at most $36$ rational points.  This 
eliminates the second real Weil polynomial, and shows that 
$N_9(12)\le 61$.  \qed

The best lower bound on $N_9(12)$ that we know at this time is $56$, as
shown by Gebhardt~\cite{Gebhardt}*{Tbl.~2, p.~96}, so there is still a
gap between our current lower and upper bounds.

\subsection{New values of $N_q(g)$.}
Running our new program, we find that $N_4(7) \le 21$ and 
$N_8(5)\le 29$.  Niederreiter and Xing~\cite{NiederreiterXing} showed
that there is a genus-$7$ curve over $\FF_4$ with $21$ points, and van
der Geer and van der Vlugt~\cite{GeerVlugt1997} showed that there is a
genus-$5$ curve over $\FF_8$ with $29$ points, so we see that
$N_4(7) = 21$ and $N_8(5) = 29$.

Let us sketch how our new program was able to improve upon the earlier
program to show that $N_4(7) \le 21$.  The earlier program showed that
any genus-$7$ curve over $\FF_4$ with $22$ points must have one of the
following five real Weil polynomials:
\begin{align*}
    &x (x + 2)^2 (x + 3)^3 (x + 4),\\
    &(x + 3)^3 (x^4 + 8 x^3 + 20 x^2 + 16 x + 1),\\
    &(x + 1) (x + 3)^4 (x^2 + 4 x + 1),\\
    &(x + 1) (x + 3)^3 (x + 4) (x^2 + 3 x + 1),\\
    &(x + 2)^3 (x + 4)^2 (x^2 + 3 x + 1).
\end{align*}
Our new program eliminates these possibilities.  The first real Weil
polynomial is forbidden by an argument on Hermitian forms.  The second
and fifth are eliminated by Theorem~\ref{T:GE}(b); one can show that
curves with these real Weil polynomials must be double covers of other
curves, and we obtain contradictions from Lemma~\ref{L:ramification}.
The third is eliminated for the same reason; however, for this
polynomial we need to use the gluing exponent and not just the
resultant in order to show that the curve is a double cover.  And
finally, the fourth polynomial can be eliminated by using the
supersingular factor method from Section~\ref{S:supersingular}.

\subsection{Correcting an error.}
In~\cite{HoweLauter2003}*{\S7} we attempted to show two particular
polynomials could not occur as real Weil polynomials of curves, but we
made an error, as is documented in the Corrigendum
to~\cite{HoweLauter2003}. We sketched corrected arguments in the second
appendix of the arXiv version of~\cite{HoweLauter2003}; here we provide 
all the details.

First, we would like to show that 
\[ f = (x + 2)^2 (x + 3) (x^3 + 4 x^2 + x - 3) \]
cannot be the real Weil polynomial of a genus-$6$ curve $C$ 
over~$\FF_3$.  Using Proposition~\ref{P:newec} and Table~\ref{T:8} we 
find that any curve with real Weil polynomial equal to $f$ must be a 
double cover of the unique elliptic curve $E$ over $\FF_3$ with 
trace~$-2$. But $E$ has $6$ rational points, so a double cover of $E$
can have at most $12$ points.  Since $C$ is supposed to have $15$ 
points, we see that no such curve $C$ can exist. Our new program makes
these deductions automatically.

The other argument in~\cite{HoweLauter2003}*{\S7} that we must correct
concerns genus-$4$ curves over $\FF_{27}$; we would like to show that
no such curve can have $65$ points.  Our new program shows that if 
there were a genus-$4$ curve over $\FF_{27}$ with $65$ points, it would
have to be a double cover of the unique elliptic curve $E$ over 
$\FF_{27}$ having $38$ points.  It is not hard to enumerate all of the
genus-$4$ double covers of this elliptic curve; Magma code for doing so
can be found on the first author's web site, in the section associated
to the paper~\cite{Howe2011}. The largest number of points we find on
genus-$4$ double cover of $E$ is~$64$.

\subsection{Proof that $N_{32}(4)\le 72$.}
In our earlier paper~\cite{HoweLauter2003}*{\S6.2} we showed that 
$N_{32}(4)< 75$.  Our new program shows that any genus-$4$ curve over 
$\FF_{32}$ having more than $72$ points must be a double cover of the 
unique elliptic curve with trace $-11$. Enumerating the genus-$4$ 
double covers of this curve is a feasible computational problem which
can be solved by a simple modification of the method outlined 
in~\cite{HoweLauter2003}*{\S6.2} (in general we must consider three 
possible arrangements of ramification points, but in the specific 
situation we faced in~\cite{HoweLauter2003}*{\S6.2} we could eliminate
one of these arrangements). We have implemented an algorithm to 
enumerate these double covers in Magma, and the resulting program can
be found in the file \texttt{32-4.magma}, available at the URL
mentioned in the Introduction. We find that no such double cover has
more than $71$ points, so $N_{32}(4)\le 72$.  This fact, combined with
information gleaned from \texttt{IsogenyClasses.magma}, shows that a 
genus-$4$ curve over $\FF_{32}$ having $72$ points must have real Weil
polynomial equal to $(x + 11)^2 (x^2 + 17 x + 71)$.

In 1999 Mike Zieve found a genus-$4$ curve over $\FF_{32}$ with $71$
points, and while searching for double covers of the trace $-11$ curve
we also found a number of curves with $71$ points.  For example, if 
$r\in\FF_{32}$ satisfies $r^5 + r^2 + 1 = 0$, then the genus-$4$ curve
\begin{align*}
y^2 + xy &= x^3 + x\\
z^2 + \phantom{x}z  &= (r^{14} x^2 + r^{24} x + r^{18})/(x + r) 
\end{align*}
has $71$ points.  Thus, $71\le N_{32}(4)\le 72$.

\subsection{Genus $12$ curves over $\FF_2$}
\label{SS:2-12-15}
The smallest genus $g$ for which the exact value of $N_2(g)$ is unknown
is $g = 12$; the Oesterl\'e bound is $15$, and a genus-$12$ curve with 
$14$ points is known.  Our program, plus some additional work, allows 
us to show that a genus-$12$ curve over $\FF_2$ with $15$ points must 
have one of three real Weil polynomials.

\begin{theorem}
\label{T:2-12-15}
If $C$ is a genus-$12$ curve over $\FF_2$ with $15$ rational points, 
then the real Weil polynomial of $C$ is equal to one of the following 
polynomials\textup{:}
\begin{itemize}
\item[I.]   $(x + 1)^2 (x + 2)^2 (x^2 - 2) (x^2 + 2 x - 2)^3$
\item[II.]  $(x - 1) (x + 2)^2 (x^3 + 2 x^2 - 3 x - 2) 
             (x^3 + 3 x^2 - 3) (x^3 + 4 x^2 + 3 x - 1)$
\item[III.] $(x^2 + x - 3) (x^3 + 3 x^2 - 3)
              (x^3 + 4 x^2 + 3 x - 1) (x^4 + 4 x^3 + 2 x^2 - 5 x - 3)$
\end{itemize}
\end{theorem}

\begin{proof}
Our program shows that the real Weil polynomial of such a curve is 
either one of the three listed in the theorem, or one of the following
two:
\begin{itemize}
\item[IV.] $(x^2 + x - 1)^2 (x^4 + 5 x^3 + 4 x^2 - 10 x - 11)
            (x^4 + 5 x^3 + 5 x^2 - 5 x - 5)$
\item[V.]  $(x^3 + 3 x^2 - 3) (x^4 + 5 x^3 + 5 x^2 - 5 x - 5)
            (x^5 + 4 x^4 + x^3 - 9 x^2 - 5 x + 3)$
\end{itemize}       
So to prove the theorem, all we must do is show that possibilities (IV)
and (V) cannot occur.  First we analyze possibility (IV).

Let 
\begin{align*}
h_1 &= x^2 + x - 1 \\
h_2 &= x^4 + 5 x^3 + 4 x^2 - 10 x - 11 \\
h_3 &= x^4 + 5 x^3 + 5 x^2 - 5 x - 5
\end{align*}
be the irreducible factors of the real Weil polynomial $h$ listed as 
item~(IV).  Since the constant terms of these polynomials are odd, they
correspond to isogeny classes of ordinary abelian varieties.  The 
corresponding Weil polynomials are
\begin{align*}
f_1 &= x^4 + x^3 + 3 x^2 + 2 x + 4 \\
f_2 &= x^8 + 5 x^7 + 12 x^6 + 20 x^5 + 29 x^4 + 40 x^3 + 48 x^2 + 40 x + 16 \\
f_3 &= x^8 + 5 x^7 + 13 x^6 + 25 x^5 + 39 x^4 + 50 x^3 + 52 x^2 + 40 x + 16.
\end{align*}
Let $K_1$, $K_2$, and $K_3$ be the CM fields defined by these three 
Weil polynomials, and for each $i$ let $\pi_i$ be a root of $f_i$
in~$K_i$.  For each $i$ we use Magma to check that $K_i$ has class 
number~$1$, and we compute the discriminant of the maximal order 
$\calO_i$ of~$K_i$. Using~\cite{Howe1995}*{Prop.~9.4, p.~2384} we find
that for each $i$, the discriminant of the subring 
$R_i := \ZZ[\pi_i,\pibar_i]$ of $\calO_i$ is equal to the discriminant 
of $\calO_i$, so we have $R_i = \calO_i$. Using the theory of Deligne 
modules~\cite{Howe1995}, we find for each $i$ the isogeny class with 
real Weil polynomial $h_i$ contains a unique abelian variety $A_i$.

Suppose there were a curve $C$ with real Weil polynomial equal to 
$h_1^2 h_2 h_3$.  Factoring this real Weil polynomial as $h_1^2 h_2$ 
times $h_3$, noting that the reduced resultant of $h_1h_2$ and $h_3$ 
is~$3$, and applying Lemma~\ref{L:GE}, we find that we have an exact 
sequence
\[ 0 \to \Delta \to B\times A_3 \to \Jac C \to 0,\]
where $B$ is isogenous to $A_1^2\times A_2$ and where $\Delta$ is a
self-dual group scheme that can be embedded into both $B[3]$ and 
$A_3[3]$.

Since the reduced resultant of $h_1$ and $h_2$ is $19$, there is an 
isogeny $A_1^2 \times A_2 \to B$ whose degree is a power of~$19$;
in particular, the group scheme $B[3]$ is isomorphic to
$A_1[3]^2 \times A_2[3]$.

The group scheme structure of each $A_i[3]$ is determined by the
$R_i$-module structure of $\calO_i / 3\calO_i$
(see~\cite{Howe1995}*{Lem.~4.13, p.~2372}).  Since $3$ is unramified
in $K_1$ and $K_2$, we find that $B[3]$ is a direct sum of simple group
schemes.  However, $A_3[3]$ is not semi-simple; the prime $3$ is 
ramified in $K_3$, and in fact in $K_3$ we have 
$3 = -\zeta^2 (1 - \zeta)^2$, where 
\[\zeta = -(98 + 69\, \pi_3 + 40\, \pi_3^2 + 18\, \pi_3^3 + 5\, \pi_3^4 + 
          56\,\pibar_3 + 25\, \pibar_3^2 + 7 \,\pibar_3^3) \]
is a cube root of unity. From this we see that the unique semi-simple
subgroupscheme of $A_3[3]$ is the kernel of $1 - \zeta$, so the image
of $\Delta$ under the projection from $B\times A_3$ to $A_3$ lies in 
$A_3[1-\zeta]$.

Now consider the automorphism $\alpha := (1, \zeta)$ of $B\times A_3$.
Clearly $\alpha$ acts trivially on $B\times A_3[1-\zeta]$, so it acts
trivially on the image of $\Delta$, and therefore it descends to give
an automorphism $\beta$ (of order $3$) on $\Jac C$.  Furthermore, since
$\zeta\zetabar = 1$, the automorphism $\alpha$ respects the pullback
to $B\times A_3$ of the principal polarization of $\Jac C$, so $\beta$
is an automorphism of $\Jac C$ as a polarized variety. The strong form
of Torelli's theorem~\cite{Milne1986}*{Thm.~12.1, p.~202} shows that 
$C$ must therefore have an automorphism $\gamma$ of order~$3$. Let $D$
be the quotient  $C/\langle\gamma\rangle$, so that there is a 
degree-$3$ Galois cover $C\to D$.

We calculate from its Weil polynomial that $C$ has $41$ places of 
degree~$8$.  Since $8$ is not a multiple of $3$, every degree-$8$ place
of $C$ lies over a degree-$8$ place of $D$, and since $41$ is congruent
to $2$ mod $3$, we see that at least $2$ degree-$8$ places of $D$ 
ramify in the triple cover $C\to D$. Thus, the degree of the different
of the cover is at least $32$. The Riemann--Hurwitz formula then gives
\[ 22 = 2g_C - 2 = 3(2 g_D - 2) + (\text{degree of different}) \ge 6g_D + 26,\]
so that the genus of $D$ must be negative.  This contradiction shows
that no curve over $\FF_2$ can have real Weil polynomial equal to 
possibility (IV) above.

Next we consider the polynomial $h$ from item (V) above.  As in the 
preceding case, $h$ corresponds to an ordinary isogeny class.  We will
use the results of~\cite{Howe1995} to show that there are no 
principally polarized abelian varieties in this isogeny class.  Let us
sketch what these results are and how they are used.  

A \emph{CM-order} is a ring $R$ that is isomorphic to an order in a
product of CM fields and that is stable under complex conjugation. 
Section~5 of~\cite{Howe1995} defines a contravariant functor $\calB$ 
from the category of CM-orders to the category of finite $2$-torsion 
groups.  If $\calC$ is an isogeny class of ordinary abelian varieties
corresponding to a Weil polynomial~$f$, then each irreducible factor
$f_i$ of $f$ defines a CM field $K_i$ with maximal order $\calO_i$.
Frobenius and Verschiebung generate an order $R$ in the product $K$ of
the $K_i$.  Section 5 of~\cite{Howe1995} uses properties of the 
polynomial $f_i$ to define an element $I_i$ of $\calB(\calO_i)$.  The 
map $R\subset\prod\calO_i\to\calO_i$ gives a homomorphism 
$\calB(\calO_i)\to\calB(R)$, and we define $I_\calC$ to be the sum of
the images of the $I_i$ in $\calB(R)$. Theorem~5.6 (p.~2375) 
of~\cite{Howe1995} says that there is a principally polarized variety
in $\calC$ if and only if $I_\calC=0$.

Proposition~10.1 (p.~2385) of~\cite{Howe1995} shows how to compute 
$\calB(\calO_i)$, and Proposition~10.5 (p.~2387) shows how to compute 
$\calB(R)$ and the map $\calB(\calO_i)\to \calB(R)$.  Proposition~11.3
(p.~2390) and Proposition~11.5 (pp.~2391--2392) show how to compute the
elements $I_i$ of $\calB(\calO_i)$.  We will apply these results to 
show that for the isogeny class $\calC$ defined by the ordinary real
Weil polynomial $h$ in item (V), the element $I_\calC$ is nonzero, so
that there are no principally polarized varieties in $\calC$, and in
particular no Jacobians.

Let 
\begin{align*}
h_1 &= x^3 + 3 x^2 - 3\\
h_2 &= x^4 + 5 x^3 + 5 x^2 - 5 x - 5 \\
h_3 &= x^5 + 4 x^4 + x^3 - 9 x^2 - 5 x + 3
\end{align*}
so that $h = h_1 h_2 h_3$.  The Weil polynomials corresponding to 
the $h_i$ are
\begin{align*}
f_1 &= x^6 + 3 x^5 + 6 x^4 + 9 x^3 + 12 x^2 + 12 x + 8\\
f_2 &= x^8 + 5 x^7 + 13 x^6 + 25 x^5 + 39 x^4 + 50 x^3 + 52 x^2 + 40 x + 16\\
f_3 &= x^{10} + 4 x^9 + 11 x^8 + 23 x^7 + 41 x^6 + 63 x^5 + 82 x^4 + 92 x^3 + 88 x^2 + 64 x + 32.
\end{align*}
Let $K_1$, $K_2$, and $K_3$ be the CM fields defined by these three 
Weil polynomials, and for each $i$ let $\pi_i$ be a root of $f_i$ 
in~$K_i$.  Let $\calO_i$ be the maximal order of $K_i$, and let $R$ be
the subring of $\calO_1\times \calO_2\times \calO_3$ generated by 
$\pi := (\pi_1,\pi_2,\pi_3)$ and $\pibar := (\pibar_1,\pibar_2,\pibar_3)$.

From~\cite{Howe1995}*{Prop~10.1, p.~2385} we see that 
$\calB(\calO_1) \cong \calB(\calO_3) \cong 0$ and 
$\calB(\calO_2) \cong \ZZ/2\ZZ$.  According 
to~\cite{Howe1995}*{Props.~11.3 and~11.5}, the element $I_2$ of 
$\calB(\calO_2)$ will be zero if and only if the positive square root of
$\Norm_{K_2/\QQ}(\pi_2 - \pibar_2)$ is congruent to the middle 
coefficient of $f_2$ modulo~$4$.  The norm of $\pi_2-\pibar_2$ is~$1$,
and the middle coefficient of $f_2$ is $39$, so we find that 
$I_2\ne 0$.

Let $S = \calO_1\times\calO_2\times\calO_3$.  To calculate $\calB(R)$
and the map $i^*\col\calB(S)\to\calB(R)$ obtained from the inclusion 
$i\col R\to S$, we apply~\cite{Howe1995}*{Prop.~10.5, p.~2387}.  That 
proposition shows that there is a push-out diagram
\[
\xymatrix{
D_s \ar[r]\ar[d]^{N} & \calB(S)\ar[d]^{i^*}\\
C_s \ar[r]           & \calB(R)\\
}
\]
where $D_s$ and $C_s$ are certain finite $2$-torsion groups.  Let 
$R^+$ and $S^+$ be the subrings of $R$ and $S$ consisting of elements
fixed by complex conjugation, so that 
$S^+ = \calO_1^+ \times \calO_2^+ \times \calO_3^+$ and  
$R^+ = \ZZ[\pi+\pibar]$.  Then $C_s$ has a basis (as an $\FF_2$-vector 
space) consisting of elements indexed by the set
\[\{ \text{maximal ideals $\gothp$ of $R^+$} \ | \ \ 
     \text{$\gothp$ is singular and is inert in $R/R^+$} \} \]
and $D_s$ has a basis consisting of elements indexed by the set
\[\left\{ \text{maximal ideals $\gothq$ of $S^+$} \ \left| \ \ 
     \vcenter{\hsize=1.8in\noindent
              $\gothq$ is inert in $S/S^+$,\hfill\break
              $\gothq\cap R^+$ is singular in $R^+$, and\hfill\break
              $\gothq\cap R^+$ is inert in $R/R^+$} \right.\right\}.\]
Let us compute the maximal ideals of $R^+$ that are singular and that 
are inert in $R/R^+$.  Since $R^+\cong \ZZ[x]/(h)$ and 
$R = R^+[x]/(x^2 - (\pi+\pibar)x + q)$, this is a straightforward 
matter.  We leave the details to the reader, but the only prime we find
that is singular and inert is $\gothp = (3, \pi+\pibar)$.

There are two maximal ideals of $S^+$ lying over $\gothp$: the ideal
$\gothq_1 = (3,\pi_1+\pibar_1)$ of $\calO_1^+$ and the ideal
$\gothq_3 = (3,\pi_3+\pibar_3)$ of $\calO_3^+.$  We compute that 
$\gothq_1$ splits in $\calO_1$ and that $\gothq_3$ splits in $\calO_3$.
Therefore the group $D_s$ is trivial.  It follows from the push-out 
diagram that $i^*$ is injective, so $i^*(I_2)$ is nonzero, and there is
no principally-polarized variety in the isogeny class associated 
to~$h$.
\end{proof}

\section{Bounds on Shafarevich--Tate groups}
\label{S:Sha}

Propositions~\ref{P:ec} and~\ref{P:newec} both give upper bounds on the
degrees of the smallest map from a curve $C$ to an elliptic curve~$E$,
and therefore say something about the Mordell--Weil lattice of maps
from $C$ to~$E$.  The Birch and Swinnerton-Dyer conjecture for constant
elliptic curves over one-dimensional function fields over finite 
fields, proven by Milne~\cite{Milne1968}, relates the determinants of
such Mordell--Weil lattices to certain Shafarevich--Tate groups. In
this section we make some comparisons between our results and the
conjecture of Birch and Swinnerton-Dyer, and deduce some results about
Shafarevich--Tate groups.

Let $C$ be a curve over a finite field $\Fq$ of characteristic~$p$, 
and let $K$ be its function field.  Suppose there is an embedding 
$\psi\col E\to \Jac C$ of an elliptic curve $E$ into the Jacobian 
of~$C$.  Pick a degree-$1$ divisor $X$ on $C$, and let 
$\varphi\col C\to E$ be the map associated to $\psi$ and $X$ as in 
Lemma~\ref{L:CtoE}.  The assumption that $\psi$ is an embedding
implies that $\varphi$ is \emph{minimal}; that is, $\varphi$ does not
factor through an isogeny from another elliptic curve to $E$.
(Conversely, a minimal map from $C$ to an elliptic curve gives rise to
an embedding of the elliptic curve into the Jacobian of~$C$.)

Suppose further that there is only one factor of $E$ in the Jacobian
of~$C$, up to isogeny; that is, assume that $\Jac C$ is isogenous to 
$A\times E$ for an abelian variety $A$ such that the gluing exponent
$e:=e(A,E)$ is finite.  Let $h$ be the real Weil polynomial of $A$,
let $g$ be the radical of $h$, and let $t$ be the trace of $E$.  Note
that Proposition~\ref{P:ec} says that $\deg\varphi$ divides~$e$, and 
that Proposition~\ref{P:RR} says that $e$ divides the reduced resultant
of $g$ with $x-t$, which is equal to~$g(t)$.

\begin{theorem}
\label{T:MW1}
Let $R$ be the endomorphism ring of $E$, let $\EE$ be the base 
extension of $E$ from $\Fq$ to~$K$, and let $\Sha$ be the 
Shafarevich--Tate group of $\EE$.
\begin{itemize}
\item[(a)] Suppose $R$ is an order in a quadratic field, so that we may
           write $t^2 - 4q = F^2\Delta_0$ for some fundamental
           discriminant $\Delta_0$ and conductor $F$, and so that the
           discriminant of $\End E$ is equal to $f^2\Delta_0$ for some 
           divisor~$f$ of~$F$. Then 
           \[\sqrt{\#\Sha} = \frac{F}{f}
                             \frac{\left|h(t)\right|}{\deg\varphi},\]
           and $\sqrt{\#\Sha}$ is divisible by 
           \[\frac{F}{f}\left|\frac{h(t)}{g(t)}\right|.\]
\item[(b)] Suppose $R$ is an order in a quaternion algebra. Then $q$ is
           a square, we have
           \[\sqrt{\#\Sha} = \frac{\sqrt{q}}{p}
                             \frac{\left|h(t)\right|}{\deg\varphi},\]
           and $\sqrt{\#\Sha}$ is divisible by 
           \[\frac{\sqrt{q}}{p}\left|\frac{h(t)}{g(t)}\right|.\]
\end{itemize}
\end{theorem}

\begin{proof}
Since every map from an elliptic curve isogenous to $E$ to $\Jac C$
factors through the embedding $\psi\col E\to\Jac C$, it follows that
any map from $C$ to an elliptic curve isogenous to $E$ factors 
through~$\varphi$.  In particular, the set $L$ of maps from $C$ to $E$
that take $X$ to a divisor on $E$ that sums to the identity is equal
to~$R\varphi$.

The \emph{Mordell--Weil lattice} of $\EE$ over $K$ is the group
$\EE(K)/E(\Fq)$ provided with the pairing coming from the canonical
height (see~\cite{Shioda1989}).  The natural map $L\to \EE(K)/E(\Fq)$
is a bijection, and the quadratic form on $L$ obtained from the height
pairing on $\EE(K)$ is twice the quadratic form given by the degree map
(see~\cite{SilvermanII}*{Thm.~III.4.3, pp.~217--218}).

Suppose that $R$ is an order in a quadratic field.  Using the fact that
$L = R\varphi$, we find that the determinant $D$ of the Mordell--Weil 
lattice for $\EE$ satisfies
\[D = (\deg\varphi)^2\left|\disc R\right|
    = (\deg\varphi)^2 f^2 \left|\Delta_0\right|.\]
Let the eigenvalues of Frobenius for $\Jac C$ be
\[\pi_1, \pibar_1, \pi_2, \pibar_2, \ldots, \pi_g, \pibar_g,\]
indexed so that $\pi_1$ and $\pibar_1$ are the (distinct) eigenvalues
of Frobenius for~$E$.  The Birch and Swinnerton-Dyer conjecture for 
constant elliptic curves over function 
fields~\cite{Milne1968}*{Thm.~3, pp.~100--101} says that the product 
$D\#\Sha$ is equal to 
\[               q^g
                 \left(1 - \frac{\pibar_1}{\pi_1}\right)
                 \left(1 - \frac{\pi_1}{\pibar_1}\right)
      \prod_{i>1}\Bigg(\left(1 - \frac{\pibar_i}{\pi_1}\right)
                 \left(1 - \frac{\pi_i}{\pi_1}\right)
                 \left(1 - \frac{\pibar_i}{\pibar_1}\right)
                 \left(1 - \frac{\pi_i}{\pibar_1}\right)\Bigg). \]
Combining this with the relations $\pi_i\pibar_i = q$, and using the
facts that $\pi_1 + \pibar_1 = t$ and 
$(\pi_1-\pibar_1)^2 = F^2\Delta_0$, we find that 
\begin{align*}
D\#\Sha &= -q^{1-g} (\pi_1-\pibar_1)^2 
           \prod_{i>1}\big((\pi_i-\pi_1)(\pibar_i-\pi_1)
                           (\pi_i-\pibar_1)(\pibar_i-\pibar_1)\big)\\
        &= -q^{1-g} (\pi_1-\pibar_1)^2 
           \prod_{i>1} q (\pi_1+\pibar_1 - \pi_i-\pibar_i)^2 \\
        &= F^2 \left|\Delta_0\right| \prod_{i>1}(t - \pi_i-\pibar_i)^2\\
        &= F^2 \left|\Delta_0\right| h^2(t).
\end{align*}
Using the equation for $D$ given above, we find that 
\[\sqrt{\#\Sha} = \frac{F}{f}\frac{h(t)}{\deg\varphi}.\]
As we noted earlier, $\deg\varphi$ is divisor of~$g(t)$.  This proves
statement~(a).

Suppose $R$ is a (necessarily maximal) order in a quaternion algebra.  
This implies that $q$ is a square and the Frobenius eigenvalues of $E$
are both equal to $\sqrt{q}$ or to~$-\sqrt{q}$.  Again calculating the
determinant $D$ of the Mordell--Weil lattice of $\EE$ by using the
identification $L = R\varphi$, we find that $D = (\deg\varphi)^2 p^2$;
using~\cite{Milne1968}*{Thm.~3, pp.~100--101} we find that 
\[D\#\Sha = q h^2(t).\]
From these equalities we obtain
\[\sqrt{\#\Sha} = \frac{\sqrt{q}}{p}\frac{h(t)}{\deg\varphi},\]
and as above we find that $\sqrt{\#\Sha}$ is a multiple of 
\[\frac{\sqrt{q}}{p}\frac{h(t)}{g(t)}.\qedhere\]
\end{proof}

Our results tell us something about Shafarevich--Tate groups in other
situations as well.  Suppose $C$ is a curve over a finite field $\Fq$
whose Jacobian is isogenous to $E^g$, for some ordinary elliptic curve
$E$ over $\Fq$ with trace~$t$. Let $R$ be the endomorphism ring of~$E$.
There is a universal isogeny $\Psi\col A\to\Jac C$, unique up to
isomorphism, with the property that every map $\psi\col E\to\Jac C$ 
factors through~$\Psi$;  the Deligne module for $A$ is the largest
submodule of the Deligne module for $\Jac C$ that is also an
$R$-module.

Write $t^2 - 4q = F^2\Delta_0$ for a fundamental discriminant
$\Delta_0$ and a conductor~$F$, and let $f$ be the conductor of 
$\End E$, so that $f\mid F$.  

\begin{theorem}
\label{T:MW2}
Let $K$ be the function field of $C$, let $\EE$ be the base extension
of $E$ from $\Fq$ to~$K$, and let $\Sha$ be the Shafarevich--Tate group
of $\EE$.  Then
\[ \sqrt{\#\Sha} = (F/f)^g \deg\Psi.\]
\end{theorem}

\begin{proof}
Let $\mu$ be the canonical principal polarization of $\Jac C$, and let
$\lambda$ be the pullback of $\mu$ to $A$ via the isogeny~$\Psi$.
Consider the lattice of homomorphisms $E\to\Jac C$, with the quadratic
form $Q$ provided by the square root of the degree of the pullback of 
$\mu$ to $E$. Since every homomorphism $E\to\Jac C$ factors through 
$\Psi$, this lattice is isomorphic to to the lattice of 
homomorphisms $E\to A$, with the quadratic form given by the square 
root of the degree of the pullback of $\lambda$ to $E$.  Applying 
Lemma~\ref{L:quadform}, we see that the determinant of this lattice is
\[\left|(\disc R)/4\right|^g (\deg\Psi)^2 
    = (f/2)^{2g} \left|\Delta_0\right|^g (\deg\Psi)^2.\]
As we have noted earlier, the Mordell--Weil lattice of $\EE$ is 
isomorphic to the lattice of homomorphisms $E\to\Jac C$, with the 
quadratic form~$2Q$.   Therefore the determinant $D$ of the 
Mordell--Weil lattice is 
\[D = f^{2g} \left|\Delta_0\right|^g (\deg\Psi)^2.\]
On the other hand, the Birch and Swinnerton-Dyer conjecture says that
in this case we have
\begin{align*}
D\#\Sha &=  q^g \left(1-\frac{\pibar}{\pi}\right)^g 
                \left(1-\frac{\pi}{\pibar}\right)^g \\
        &=  (\pi-\pibar)^g (\pibar-\pi)^g \\
        &= (4q - t^2)^g \\
        &= F^{2g} \left|\Delta_0\right|^g.
\end{align*}
Thus we find that
\[ \sqrt{\#\Sha} = (F/f)^g \deg\Psi. \qedhere\]
\end{proof}







\end{document}